\documentclass[12pt]{amsart}
\usepackage{amssymb,latexsym}
\usepackage{enumerate}

\newtheorem{thm}{Theorem}[section]

\newtheorem{lem}[thm]{Lemma}
\newtheorem{prop}[thm]{Proposition}

\newtheorem{conj}[thm]{Conjecture}

\newtheorem*{thmkurz}{Kurzweil's Theorem (1955)}
\newtheorem*{thmkhin}{Khintchine's Theorem (1926)}
\newtheorem*{thmgal}{Gallagher's Theorem (1962)}
\newtheorem*{thmBHKV}{Theorem BHKV (2010)}
\newtheorem*{thmKTV}{Theorem KTV (2006)}

\theoremstyle{definition}
\newtheorem{defin}[thm]{Definition}

\newtheorem*{xrem}{Remark}

\numberwithin{equation}{section}

\newcommand{\norm}[1]{\ensuremath{\left\Vert #1 \right\Vert}}

\newcommand{\infabs}[1]{\ensuremath{\left\vert #1 \right\vert}}

\newcommand{\Bad}{\text{\textup{\bf Bad}}}

\begin{document}

\baselineskip=17pt

\title[T.I.D.A. \& badly approximable sets]{Twisted inhomogeneous Diophantine approximation and badly approximable sets}

\author[S. Harrap]{Stephen Harrap}
\address{Department of Mathematics, University of York,
  Heslington, York, YO10 5DD, United Kingdom}
\email{sgh111@york.ac.uk}

\date{}

\begin{abstract}
	For any real pair $i, j \geq 0$ with $i+j=1$ let $\Bad(i, j)$ denote the set 
	of $(i, j)$-badly approximable pairs. That is, $\Bad(i, j)$ consists of
	irrational vectors $\mathbf{x}:=(x_1, x_2) \in \mathbb{R}^2$ for which there 
	exists a positive constant $c(\mathbf{x})$ such that
	\begin{equation*}
		\max \left\{\norm{qx_1}^{1/i}, \norm{qx_2}^{1/j} \right\} > 	
		c(\mathbf{x})/q \quad \quad \forall \, q \in \mathbb{N}.
	\end{equation*}
	A new characterization of
	$\Bad(i, j)$ in terms of `well-approximable' vectors in the area of 	
	`twisted' inhomogeneous Diophantine approximation is established.
	In addition, it is shown that $\Bad^{\mathbf{x}}(i, j)$, the `twisted' 	
	inhomogeneous analogue of $\Bad(i, j)$, has full Hausdorff dimension~$2$ 	
	when $\mathbf{x}$ is chosen from $\Bad(i, j)$. The main results naturally 	
	generalise the $i=j=1/2$ work of Kurzweil.
\end{abstract}



\maketitle

\section{Introduction}
\label{sec:introduction}

\subsection{Background -- the homogeneous theory}
\label{sec:Background1}

A classical result due to Dirichlet states that for any real number $x$ there exist infinitely many natural numbers
$q$ such that
\begin{equation}
	\label{eqn:dirichlet}
	\norm{qx}\leq \frac{1}{q},
\end{equation}
where $\norm{\cdot}$ denotes the distance to the nearest integer. This result can easily be generalised to higher
dimensions. In particular, the following `weighted' simultaneous version of the above statement is valid.
Choose any positive real numbers $i$ and $j$ satisfying
\begin{equation}
	\label{eqn:ijconditions}
	i, j \, \geq \, 0 \quad \text{ and } \quad i+j \, = \, 1.
\end{equation}
Then, for any vector $\mathbf{x} \in
\mathbb{R}^2$ there exist infinitely many natural numbers $q$ such that
\begin{equation}
	\label{eqn:wsdirichlet}
	\max \left\{\norm{qx_1}^{1/i}, \norm{qx_2}^{1/j} \right\} \leq \frac{1}{q}.
\end{equation}
Here, without loss of generality, if $i=0$ we employ the convention that $\norm{x}^{1/i}=0$ and so
the above statement reduces to Dirichlet's original result.
It is natural to ask whether the right hand side of inequality (\ref{eqn:wsdirichlet}) can in general be tightened;
that is, if
$1/q$ may be replaced by $c/q$ for some absolute constant $c \in (0, 1)$ whilst still allowing
(\ref{eqn:wsdirichlet}) to hold infinitely often for all real vectors.  It is still an open problem as to
whether there exists an `optimal' constant in the sense that the statement
holds only finitely often for at least one real vector if it is replaced by any lesser constant.
Conversely, in the one-dimensional setting, concerning statement (\ref{eqn:dirichlet}), such an `optimal' constant (namely $1/\sqrt{5}$) was found by
Hurwitz (e.g., Theorems 193 \& 194 in \cite[Chapter XI]{HW}).

The above discussion motivates the study of real vectors $\mathbf{x}$ for which the right hand side of (\ref{eqn:wsdirichlet}) cannot be improved by an
arbitrary positive constant. Throughout, we will impose the following natural restriction on these vectors.
We say $\mathbf{x}:=(x_1, x_2)$ is \textit{irrational} (abbreviated \textit{irr.}) if its components $x_i$ together with 1 are linearly independent over the rationals. 
\begin{defin}
	An irrational vector  $\mathbf{x}$ is \textit{$(i, j)$-badly approximable} if there exists a constant $c(\mathbf{x})>0$ such that
	\begin{equation*}
		\max \left\{\norm{qx_1}^{1/i}, \norm{qx_2}^{1/j} \right\} \, \, > \, \, \frac{c(\mathbf{x})}{q} \quad \quad \forall \, q
		\in \mathbb{N}.
	\end{equation*}
	The set of all such vectors will be denoted $\Bad(i, j)$.
\end{defin}
The results of this paper (for $i,j>0$) do remain true when $\mathbf{x}$ is not assumed to be irrational in the above and later definitions. However,
we choose to avoid this degenerate case for the sake of clarity.
Furthermore, all of the sets and arguments considered in this paper are invariant under integer translation,
so there will be no loss in generality in assuming throughout that all vectors are confined to the
unit square (or the unit $n$-cube when in higher dimensions) unless otherwise stated.
Accordingly, for example, if $i=0$ the
set $\Bad(0, 1)$ will be identified with $[0,1]\times \Bad$, where $\Bad$ is the standard one dimensional set of badly
approximable numbers. In other words, $\Bad(0, 1)$ consists of vectors $\mathbf{x}$ with $x_1 \in [0,1]$ and
\begin{equation*}
	x_2 \in \Bad:=\left\{\text{irr. } x \in [0, 1]: \, \exists \, c(x)>0 \text{ s.t.} \norm{qx} > \frac{c(x)}{q} \quad
	\forall \, q \in \mathbb{N} \right\}.
\end{equation*}
\begin{defin}
	A mapping $\psi : \mathbb{N} \rightarrow \mathbb{R}$ is an \textit{approximating function} if $\psi$ is strictly positive
	and non-increasing.
\end{defin}
\begin{defin}
	For any approximating function $\psi$, define $\textbf{W}_{(i, j)}(\psi)$ to be the set of 
	vectors $\mathbf{x} \in [0, 1]^2$ such that the inequality
	\begin{equation*}
		\max \left\{\norm{qx_1}^{1/i}, \norm{qx_2}^{1/j} \right\} \, \, \leq \, \,  \psi(q)
	\end{equation*}
	holds for infinitely many natural numbers $q$.
\end{defin}

Application of the following classical theorem of Khintchine \cite{Khi24} yields that for every pair of reals $i$, $j$ satisfying (\ref{eqn:ijconditions}) the set $\Bad(i, j)$ is of two-dimensional
Lebesgue measure zero.
Lebesgue measure will hereafter be denoted $\mu$.

\begin{thmkhin}
	For any pair of reals $i$, $j$ satisfying (\ref{eqn:ijconditions}) and any approximating function $\psi$ we have
	\begin{equation*}
		\mu \left( \textbf{W}_{(i, j)}(\psi) \right) =
		\begin{cases}
			0, & \displaystyle\sum_{r=1}^{\infty} \psi(r) \, < \, \infty. \\
  		& \\
  		1, & \displaystyle\sum_{r=1}^{\infty} \psi(r) \, = \, \infty.
  	\end{cases}
	\end{equation*}
\end{thmkhin}
\noindent
It is worth emphasising here that the choice of approximating function $\psi$ is completely irrelevant once the reals $i, j$ have been fixed. We also mention here that in the  $i=j=1/2$ case the monotonicity restriction imposed on $\psi$ can be relaxed (see \cite{Gal2} for details). However, whether this is true in general is still an 
open problem.

The question of whether
each null set $\Bad(i, j)$ is non-empty was formally\footnote[2]{The arguements used by Davenport in \cite{Dav} to show that $\Bad(1/2, 1/2)$ is uncountable
can easily be adapted to show that $\Bad(i, j)$ is uncountable for every choice of reals $i, j$ satisfying (\ref{eqn:ijconditions}).}
 answered by Pollington \& Velani \cite{PV}
who showed that for every choice of reals $i$, $j$ satisfying (\ref{eqn:ijconditions}) we have
\begin{equation}
	\label{eqn:pv}
	\dim \left( \Bad(i, j) \cap \Bad(1, 0) \cap \Bad(0, 1) \right) = \dim \left( [0,1]^2 \right) = 2.
\end{equation}
Here, and throughout, `$\dim$' denotes standard Hausdorff dimension.
With this result in mind, the aim of this paper is to obtain an expression for $\Bad(i, j)$ in terms of
`well-approximable' vectors in the area of `twisted' inhomogeneous Diophantine approximation.

\subsection{Background -- the `twisted' inhomogeneous theory}
\label{sec:Background2}

Another result of Khintchine (see for example \cite[Chapter $10$, Theorem $10.2$]{Hua}) states that for any irrational $x$ and any real $\gamma$ there exist infinitely many
natural numbers $q$ such that
\begin{equation}
	\label{eqn:khintchine}
	\norm{qx-\gamma}\leq \frac{1+\epsilon}{\sqrt{5}q},
\end{equation}
where $\epsilon>0$ is an arbitrary constant. The inequality is `optimal' and differs from Hurwitz's homogeneous
`$\gamma=0$' theorem by only the constant $\epsilon$. When certain restrictions are placed
on the choice of $\gamma$, a tighter `optimal' inequality was found to hold by Minkowski \cite{31.0213.02}:
The right hand side of (\ref{eqn:khintchine}) can be replaced with $1/(4q)$ if it
is assumed that $\gamma$ is not of the form $\gamma=mx+n$ for some integers $m$ and $n$. Both of these statements
lead to the implication that the sequence $\left\{qx \right\}_{q \in \mathbb{N}}$ modulo one is dense in the
unit interval for any irrational $x$.
Moreover, Kronecker's Theorem (see \cite{Kron}) implies that the sequence
$\left\{q\mathbf{x} \right\}_{q \in \mathbb{Z}}$ modulo one is dense in $[0,1]^2$ for any irrational vector
$\mathbf{x}$. Furthermore, the sequence is uniformly distributed.
This naturally leads to the concept of approximating real vectors
$\boldsymbol\gamma$ in $[0, 1]^2$ by the sequence $\left\{q\mathbf{x} \right\}_{q \in \mathbb{N}}$ modulo one with increasing
degrees of accuracy. For obvious reasons we call this approach `twisted' Diophantine approximation. 

\begin{defin}
For each fixed approximating function $\psi$, any irrational vector $\mathbf{x}$ and each pair $i$, $j$ satisfying (\ref{eqn:ijconditions})
 define $\textbf{W}_{(i, j)}^{\mathbf{x}}(\psi)$ to be the set of vectors $\boldsymbol\gamma:=(\gamma_1, \gamma_2) \in [0,1]^2$
such that the inequality
\begin{equation*}
	\max\left\{ \norm{qx_1-\gamma_1}^{1/i}, \norm{qx_2-\gamma_2}^{1/j} \right\} \leq \psi(\infabs{q})
\end{equation*}
holds for infinitely many non-zero integers $q$.
\end{defin}

Establishing a Khintchine-type law (an analogue to Khintchine's Theorem) for the Lebesgue measure of
$\textbf{W}_{(i, j)}^{\mathbf{x}}(\psi)$ is more difficult than in the homogeneous case. That said, by
utilising the Borel-Cantelli lemma from probability theory it is easy to show that for every $i$, $j$ satisfying
(\ref{eqn:ijconditions}), any irrational $\mathbf{x}$ and every
approximating function~$\psi$ we have
\begin{equation*}
	\mu \left( \textbf{W}_{(i, j)}^{\mathbf{x}}(\psi) \right) = 0 \quad \text{if} \quad
	\displaystyle\sum_{r=1}^{\infty} \psi(r) \, < \, \infty.
\end{equation*}
One might therefore expect that no matter what the choice of reals $i$, $j$, irrational $\mathbf{x}$ or approximating function $\psi$ we
should be able to conclude that $\mu \left( \textbf{W}_{(i, j)}^{\mathbf{x}}(\psi) \right) = 1$ if the above sum
diverges. However, the following statement, a consequence of Theorem~\ref{thm:doublymetric} (see Appendix),
suggests that once the reals $i$, $j$ have been fixed the set of irrational vectors
for which we do obtain a set of full measure is dependent on the choice of approximating function. This subtle distinction
is what makes the metrical theory in the `twisted' setting more delicate, and sophisticated, than its standard homogeneous counterpart.

\begin{thm}[Twisted Khintchine-type Theorem]
	\label{thm:fullmeasure}
	Let $\psi$ be a fixed approximating function. Then, for $\mu$-almost all irrational vectors
	$\mathbf{x} \in [0,1]^2$
	\begin{equation*}
		\mu \left( \textbf{W}_{(i, j)}^{\mathbf{x}}(\psi) \right) = 1 \quad \text{if} \quad
		\displaystyle\sum_{r=1}^{\infty} \psi(r) \, = \, \infty.
	\end{equation*}
\end{thm}

Approximating functions whose sum diverges will hereafter simply be referred to as \textit{divergent} and the set of all
divergent approximating functions will be denoted by $\mathcal{D}$.

\begin{defin}
	Fix a pair of reals $i$, $j$ satisfying (\ref{eqn:ijconditions}). Then, for each $\psi \in \mathcal{D}$ we define
	\begin{equation*}
		\textbf{V}_{(i, j)}(\psi):=\left\{ \text{irr. } \mathbf{x}: \mu \left(
		\textbf{W}_{(i, j)}^{\mathbf{x}}(\psi) \right)=1 \right\}.
	\end{equation*}
\end{defin}

Note that Theorem~\ref{thm:fullmeasure} is equivalent to the statement ``$\mu \left( \textbf{V}_{(i, j)}(\psi) \right)=1$ for each
$\psi\in \mathcal{D}$''.
In view of this theorem we ask whether there exist irrational
vectors $\mathbf{x}$ such that a set of full measure is obtained irregardless of the choice of divergent approximating function.
In other words,
we wish to characterise the set
\begin{equation*}
	\bigcap_{\psi \in \mathcal{D}} \textbf{V}_{(i, j)}(\psi).
\end{equation*}
It is certainly not obvious as to whether the intersection is non-empty.
Almost all activity in the past has been
centred on the specific $i=j=1/2$ case where elements of $\Bad\left(1/2,1/2 \right)$ are commonly
referred to as \textit{simultaneously badly approximable pairs}.
The most notable breakthrough was made by Kurzweil \cite{MR0073654}, who proved the following remarkable result.
\begin{thmkurz}
	\label{thm:kurzorig}
	\begin{equation*}
		\bigcap_{\psi \in \mathcal{D}} \textbf{V}_{\left( \frac{1}{2},\frac{1}{2} \right)}(\psi) =
		\textstyle{\Bad\left( \frac{1}{2},\frac{1}{2} \right)}.	
	\end{equation*}
\end{thmkurz}

In fact, Kurzweil's result was more general than the above (see \S\ref{sec:HigherDimensions} for further discussion)
but did not touch upon the weighted route with which we are interested.
His work has since been extended in various directions by Fayad \cite{MR2268368} (who gave a shorter proof of the above result from a dynamical systems viewpoint),
Tseng \cite{Tse08} and Chaika \cite{Cha}.

The work of Kim \cite{MR2335077} in a similar vein inspired activity concerning real vectors that are badly
approximable in the `twisted' inhomogeneous sense.

\begin{defin}
Fix an irrational vector $\mathbf{x}\in [0,1]^2$ and two real numbers $i$ and $j$ satisfying (\ref{eqn:ijconditions}). Define
$\Bad^{\mathbf{x}}(i, j)$ as the set of vectors $\boldsymbol\gamma \in [0,1]^2$ for
which there exists a constant $c(\boldsymbol\gamma)>0$ such that
\begin{equation*}
	\max \left\{\norm{qx_1-\gamma_1}^{1/i}, \norm{qx_2-\gamma_2}^{1/j} \right\} >
	\frac{c(\boldsymbol\gamma)}{\infabs{q}}	\quad \quad \quad \text{ for all }  \, q \in \mathbb{Z}_{\neq 0}.
\end{equation*}
\end{defin}

The set $\Bad^{\mathbf{x}}(i, j)$ represents the twisted inhomogeneous analogue of $\Bad(i, j)$
introduced in \S\ref{sec:Background1}. Previous work has again been confined to the $i=j=1/2$
setting. In particular, Bugeaud et al \cite{BHKV} proved the following result
(also see the work of Tseng \cite{Tse09} and Moshcheivitin \cite{Mos09} for more recent extensions).
\begin{thmBHKV}
	For any irrational $\mathbf{x} \in [0,1]^2$,
	\begin{equation*}
		\dim\left(\Bad^{\mathbf{x}}\left(\frac{1}{2},\frac{1}{2}\right) \right) \, = \, 2.
	\end{equation*}
\end{thmBHKV}

Once more, the statement proved was more general than the above, which has been simplified for our needs.
At the time of writing there were no known results concerning the
Hausdorff dimension of $\Bad^{\mathbf{x}}(i, j)$ for a general pair $i$ and $j$.

\section{The main results}
\label{sec:statement-results}

\subsection{Statements of Results}
\label{sec:StatementsOfResults}

The following statement represents our main theorem and generalises Kurzweil's Theorem
from the classical `$1/2$--$1/2$' statement to all `$(i, j)$--weightings'.
\begin{thm}
	\label{thm:kurz}
	For every pair of reals $i$ and $j$ satisfying (\ref{eqn:ijconditions}),
	\begin{equation*}
		\bigcap_{\psi \in \mathcal{D}} \textbf{V}_{(i,j)}(\psi) = \Bad(i, j).	 
	\end{equation*}
\end{thm}
In view of Khintchine's Theorem and statement (\ref{eqn:pv}), Theorem~\ref{thm:kurz} immediately
implies that the intersection on the LHS above is of $2$-dimensional Lebesgue measure zero and of full
Hausdorff dimension two.

Our next result makes a contribution towards determining the Hausdorff dimension of $\Bad^{\mathbf{x}}(i, j)$.
\begin{thm}
	\label{thm:dimbad}
	For any real $i$ and $j$ satisfying (\ref{eqn:ijconditions}) and any $\mathbf{x} \in \Bad(i, j)$,
	\begin{equation*}
			\dim \left(\Bad^{\mathbf{x}}(i, j) \right) \, = \,  2.
	\end{equation*}
\end{thm}

The proof of this theorem makes use of a general framework developed by Kristensen, Thorn \& Velani
\cite{MR2231044}. This framework was
designed for establishing dimension results for large classes of badly approximable sets and the above statement
constitutes one further application. In all likelihood the above result is true without the assumption on $\mathbf{x}$. 
\begin{conj}
	\label{conj:harrap}
	For any real $i$ and $j$ satisfying (\ref{eqn:ijconditions}) and any irrational vector $\mathbf{x}\in [0,1]^2$,
	\begin{equation*}
		\dim\left( \Bad^{\mathbf{x}}(i, j) \right) \, = \, 2.
	\end{equation*}
\end{conj}

It seems that the ideas of \cite{BHKV}, which also make
use of the framework in \cite{MR2231044}, are not extendable to the full weighted setting of Conjecture 
\ref{conj:harrap}; a new approach may be required.
Note that Theorem~\ref{thm:dimbad}, together with (\ref{eqn:pv}) trivially implies that the conjecture is true for a set of irrational vectors
$\mathbf{x}$ of full dimension.
\begin{xrem}
	Since submission, Nikolay Moshchevitin and the named auther have strengthened Theorem \ref{thm:dimbad} from 
	a statement implying full Hausdorff dimension to the statement that $\Bad^{\mathbf{x}}(i, j)$ is `winning' 
	under the given conditions. However, obtaining a solution to Conjecture \ref{conj:harrap} still remains out of reach.
\end{xrem}

\subsection{Higher Dimensions}
\label{sec:HigherDimensions}

We describe the $n$-dimensional generalisation of the sets
$\Bad(i, j)$ and $\textbf{V}_{(i,j)}(\psi)$
along with the higher dimensional analogue of the statements in \S\ref{sec:StatementsOfResults}.
Fix any $n$-tuple of reals $\mathbf{i}:=i_1, \ldots, i_n \geq 0$ such that $\sum_{j=1}^n i_j=1$. We naturally define
$\Bad(\mathbf{i})$ to be the set of vectors $\mathbf{x}:=(x_1, \ldots, x_n) \in [0,1]^n$
for which there exists a constant $c(\mathbf{x})>0$ such that
\begin{equation*}
	\max \left\{ \norm{qx_1}^{1/i_1}, \ldots, \norm{qx_n}^{1/i_n} \right\} > \frac{c(\mathbf{x})}{q} \quad \quad \quad
	\forall \, q \in \mathbb{N}.
\end{equation*}
For any approximating function $\psi$ and any irrational vector $\mathbf{x}\in [0,1]^n$, we denote by
$\textbf{W}_{\mathbf{i}}^{\mathbf{x}}(\psi)$ the set of vectors $\boldsymbol\gamma:=(\gamma_1, \ldots, \gamma_n)
\in [0, 1]^n$ such that
\begin{equation*}
	\max\left\{ \norm{qx_1-\gamma_1}^{1/i_1}, \ldots, \norm{qx_n-\gamma_n}^{1/i_n} \right\} \, \leq \,
	\psi(\infabs{q})
\end{equation*}
for infinitely many non-zero integers $q$.
Also, set
\begin{equation*}
	\textbf{V}_{\mathbf{i}}(\psi):=\left\{ \mathbf{x} \in [0,1]^n: \mu_n \left(
	\textbf{W}_{\mathbf{i}}^{\mathbf{x}}(\psi) \right)=1 \right\},
\end{equation*}
where $\mu_n$ denotes the standard $n$-dimensional Lebesgue measure, and once more denote by 
$\mathcal{D}$ the set of approximating functions for which
\begin{equation*}
	\displaystyle\sum_{r=1}^{\infty} \psi(r) \, = \, \infty.
\end{equation*}
The proof of Theorem~\ref{thm:kurz} can be extended in the obvious way, with no new ideas or difficulties, allowing us to establish the following statement.
For every real $n$-tuple $\mathbf{i}$ such that $i_1, \ldots, i_n \geq 0$ and $\sum_{j=1}^n i_j=1$,
\begin{equation}
	\label{eqn:nkurz}
	\bigcap_{\psi \in \mathcal{D}} \textbf{V}_{\mathbf{i}}(\psi) = \Bad(\mathbf{i}).	
\end{equation}
Khintchine's Theorem and statement (\ref{eqn:pv}) can also be generalised and yield that the above
intersection is of $n$-dimensional Lebesgue measure zero and of full Hausdorff dimension~$n$.
As eluded to above, Kurzweil proved in \cite{MR0073654} that equality (\ref{eqn:nkurz}) holds in the case that $i_1=\cdots=i_n=1/n$, for every natural number $n$.
This includes the one-dimensional formulation of the problem corresponding to the set $\Bad$.
However, in these generalisations the notation gets rather awkward and so for the sake of clarity (and relevance to the material in \S\ref{sec:introduction})
we will prove the `$n=2$' case only.

The set $\Bad^{\mathbf{x}}(\mathbf{i})$ can be defined in the obvious way and analogues to
Theorem~\ref{thm:dimbad} and Conjecture \ref{conj:harrap} can easily be established.
The framework and proof of Theorem~\ref{thm:dimbad} in \S\ref{sec:ProofOfTheoremRefThmDimbad} can easily be modified to establish the corresponding
result in higher dimensions.

\section{Multiplicative Diophantine Approximation}
\label{sec:multiplicative-diophantine-approximation}

This section comprises of a brief discussion of related problems in the area of multiplicative Diophantine approximation, where loosely speaking
the supremum norm is replaced by the geometric mean. For example, one could consider the set of vectors that are `well approximable' in a multiplicative sense.

\begin{defin}
	Let $\psi$ be any approximating function. Then, define
	\begin{equation*}
		\textbf{W}_{M}(\psi):= \left\{ \mathbf{x} \in [0,1]^2: \norm{qx_1} \norm{qx_2} \leq \psi(q) \, \text{ for inf.\ many } q \in \mathbb{N} \right\}.
	\end{equation*}
\end{defin}
\noindent
The relevant measure-theoretic result concerning $\textbf{W}_{M}(\psi)$ was found by Gallagher \cite{Gal} who proved a theorem implying the following.

\begin{thmgal}
	For any approximating function $\psi$,
	\begin{equation*}
		\mu \left( \textbf{W}_{M}(\psi) \right) \, = \,
		\begin{cases}
			0, & \displaystyle\sum_{r=1}^{\infty} \psi(r) \log (1/\psi(r))  \, < \, \infty. \\
			& \\
  		1, & \displaystyle\sum_{r=1}^{\infty} \psi(r) \log (1/\psi(r)) \, = \, \infty.
  	\end{cases}
	\end{equation*}
\end{thmgal}

It is natural to develop a twisted theory for the multiplicative setup.

\begin{defin}
	Fix any approximating function $\psi$ and any irrational vector $\mathbf{x}$ in $[0,1]^2$. Then, define
	\begin{equation*}
		\textbf{W}_{M}^{\mathbf{x}}(\psi):= \left\{ \mathbf{\gamma} \in [0,1]^2: \norm{qx_1-\gamma_1} \norm{qx_2-\gamma_1} \leq \psi(\infabs{q}) \, \text{ for inf.\ }
		q \in \mathbb{Z}_{\neq 0} \right\}.
	\end{equation*}
\end{defin}
\noindent
The following statement is a consequence of Theorem~\ref{thm:doublymetric} (see the Appendix).
\begin{thm}
	\label{thm:multmetrical}
	Fix any approximating function $\psi$. Then for $\mu$-almost all irrational vectors $\mathbf{x} \in [0,1]^2$ we have
	\begin{equation*}
		\mu \left( \textbf{W}_{M}^{\mathbf{x}}(\psi) \right) \, = \,
		\begin{cases}
			0, & \, \displaystyle\sum_{r=1}^{\infty} \psi(r) \log (1/\psi(r))  \, < \, \infty. \\
			& \\
  		1, & \, \displaystyle\sum_{r=1}^{\infty} \psi(r) \log (1/\psi(r)) \, = \, \infty.
  	\end{cases}
	\end{equation*}
\end{thm}

Once more one could ask whether there exist irrational vectors $\mathbf{x}$ such that a set of full measure is obtained irrespective of the
choice of approximating function. Accordingly, let $\mathcal{D}_M$ denote the set of approximating functions for which $\sum_{r=1}^{\infty} \psi(r) \log (1/\psi(r))$ diverges and
define
\begin{equation*}
	\textbf{V}_{M}(\psi):=\left\{ \text{irr. } \mathbf{x}: \mu \left(
	\textbf{W}_{M}^{\mathbf{x}}(\psi) \right)=1 \right\}.
\end{equation*}
Consider the intersection
\begin{equation}
	\label{eqn:mintersect}
	\bigcap_{\psi \in \mathcal{D}_M} \textbf{V}_{M}(\psi).
\end{equation}
\noindent
In view of Theorem~\ref{thm:kurz}, one might expect that (\ref{eqn:mintersect}) is equivalent to the multiplicative analogue of the set of badly approximable pairs.
However, quite how such an analogue should be defined is up for debate.

One could argue
that a valid choice for a set of \textit{multiplicatively badly approximable numbers} might be
\begin{equation*}
	\Bad_L:=\left\{\mathbf{x} \in [0,1]^2: \, \exists \, c(\mathbf{x})>0 \text{ s.t.}
	\norm{qx_1}\norm{qx_2} > \frac{c(\mathbf{x})}{q} \quad \forall \, q \in \mathbb{N} \right\}.
\end{equation*}
The famous Littlewood conjecture states that the set $\Bad_L$ is empty.
For recent developments and background concerning the Littlewood conjecture see \cite{EKL}, \cite{PV2} and the references therein.

Another candidate for the multiplicatively badly approximable numbers is the larger set
\begin{equation*}
	\text{\textup{\bf Mad}}:=\left\{\mathbf{x} \in [0,1]^2: \, \exists \, c(\mathbf{x})>0 \text{ s.t.}
	\norm{qx_1}\norm{qx_2} > \frac{c(\mathbf{x})}{q \log{q}} \quad \forall \, q \in \mathbb{N} \right\},
\end{equation*}
recently introduced in \cite{mad}. Hence, the following question arises:
\begin{equation*}
Can   \ 	\bigcap_{\psi \in \mathcal{D}_M} \textbf{V}_{M}(\psi) \ \text{{\em be characterized as}}  \  \Bad_L \, \text{ \em or } \, \text{\textup{\bf Mad}}  ?
\end{equation*}
Even establishing that $\Bad_L \subseteq \bigcap_{\psi \in \mathcal{D}_M} \textbf{V}_{M}(\psi)$ seems non-trivial.

\section{Proof of Theorem~\ref{thm:kurz}}
\label{sec:ProofsOfTheorems}

\subsection{Proof of Theorem~\ref{thm:kurz} (Part $1$)}
\label{sec:proof-of-theorem-kurz-part-1}

If either $i=0$ or $j=0$
the theorem simplifies to a one-dimensional `$n=1$' version of Kurzweil's Theorem
corresponding to $\Bad$. Therefore, we can and will assume hereafter that $i, j > 0$.
The proof of Theorem~\ref{thm:kurz} takes the form of two inclusion propositions, the first of
which is proved in this section.
\begin{prop}
	\label{prop:kurz1}
	For every real $i,j > 0$ such that $i+j=1$,
	\begin{equation*}
		\label{eqn:kurz1}
		\bigcap_{\psi \in \mathcal{D}} \textbf{V}_{(i,j)}(\psi) \subseteq \Bad(i, j).	
	\end{equation*}
\end{prop}

\begin{proof}
	
	We will show that if $\mathbf{x} \notin \Bad(i, j)$ then $\mathbf{x}
	\notin \bigcap_{\psi \in \mathcal{D}} \textbf{V}_{(i,j)}(\psi)$
	and prove the result via a contrapositive argument.
	In particular, we will show that for every such $\mathbf{x}$ there exists an approximating
	function $\psi_0 \in \mathcal{D}$ for which
	\begin{equation}
		\label{eqn:lemma1plan}
		\mu \left( \textbf{W}_{(i, j)}^{\mathbf{x}}(\psi_0) \right) =0;
	\end{equation}
	i.e., the points $\boldsymbol\gamma:=(\gamma_1, \gamma_2) \in [0,1]^2$ that
	satisfy the inequality
	\begin{equation*}
		\max\left\{ \norm{qx_1-\gamma_1}^{1/i}, \norm{qx_2-\gamma_2}^{1/j} \right\} \, \leq \, \psi_0(\infabs{q})
	\end{equation*}
	for infinitely many non-zero integers $q$ form a null set with respect to the Lebesgue measure.
	
	First, if $\mathbf{x} \notin \Bad(i, j)$ then by definition there exists a sequence
	$\left\{q_k \right\}_{k \in \mathbb{N}}$
	of non-zero integers such that
	\begin{equation}
		\label{eqn:badsequence}
		\max\left\{ \norm{q_k x_1}^{1/i}, \norm{q_k x_2}^{1/j} \right\} <
		\frac{c_k}{\infabs{q_k}}, \quad \quad \quad  \infabs{q_k}< \infabs{q_{k+1}} \text{  }
		\, \forall \text{  } k \in \mathbb{N},
	\end{equation}
	where $c_k>0$ and $c_k\rightarrow 0$ as $k\rightarrow \infty$.
	Furthermore, it can be assumed that
	\begin{equation}
		\label{eqn:ckcondition}
		1 > c_k >  2^{3/(2\min\left\{i, j \right\})}  c_{k+1} \quad \text{  } \forall \text{  } k \in \mathbb{N}.
	\end{equation}
	If this were not the case then we could simply choose a suitable subsequence of $\left\{q_k \right\}$. In addition, it
	may also be assumed that the sequence $\left\{\left( c_k \right)^{-1/3} \right\}_{k \in \mathbb{N}}$ takes integer values for every index $k$.
	Note that the latter assumption, along with condition (\ref{eqn:ckcondition}),	guarantees that for every $k$
	\begin{equation}
		\label{eqn:ckproperty}
		\left( c_k \right)^{-\frac{1}{3}} \geq 2.
	\end{equation}
	
	We wish to construct a divergent approximating function $\psi_0$
	for which equation (\ref{eqn:lemma1plan}) is fulfilled. To that end, we introduce some useful notation.
	For each $k \geq 1$, let $n_k:=\infabs{q_k}(c_k)^{-1/3}$. In
	view of the above assumptions the sequence $\left\{n_k\right\}_{k \in \mathbb{N}}$ is increasing and
	takes strictly positive integer values for each index $k$.
	That said, we set $n_0:=0$ for future conciseness. Next, for each natural number $r$ define
	\begin{equation*}
		\psi_0(r): =
		\begin{cases}
			1, & \quad r \leq n_1. \\
  		\infabs{q_{k+1}}^{-1} \left( c_{k+1} \right)^{\frac{1}{3}}, & \quad
  		n_{k} < r \leq n_{k+1} \quad \quad \text{ for every } k \geq 1.
  	\end{cases}
	\end{equation*}
	It is obvious that $\psi_0$ is a decreasing and strictly positive function. To show $\psi_0 \in \mathcal{D}$, note that
	\begin{eqnarray*}
		\sum_{r=1}^{\infty} \psi_0(r)  & > & \sum_{k=1}^{\infty} \sum_{r=n_k+1}^{n_{k+1}}
		\psi_0(r)
		\\	& = & \sum_{k=1}^{\infty} \left( n_{k+1} - (n_k+1)+1 \right) \psi_0(n_{k+1}) \\ &
		= & \sum_{k=1}^{\infty} \left(\infabs{q_{k+1}} \left( c_{k+1} \right)^{-\frac{1}{3}} -
		 \infabs{q_{k}} \left( c_{k} \right)^{-\frac{1}{3}} \right) 
		\infabs{q_{k+1}}^{-1} \left( c_{k+1} \right)^{\frac{1}{3}}  \\
		& = & \sum_{k=1}^{\infty} \left( 1 - \frac{\infabs{q_{k}}}{\infabs{q_{k+1}}} \left( \frac{c_{k+1}}{c_{k}}
		\right) ^{\frac{1}{3}} \right) \\
		& > & \sum_{k=1}^{\infty} \left( 1 - \left( \frac{c_{k+1}}{c_{k}}	 \right) ^{\frac{1}{3}} \right) \quad \quad
		\quad \text{ (since $\infabs{q_{k}} < \infabs{q_{k+1}}$ )} \\
		& \stackrel{(\ref{eqn:ckcondition})}{>} & \sum_{k=1}^{\infty} \left( 1 - 2^{-1/(2 \min\left\{i, j \right\})} \right) \\
		& \geq & \sum_{k=1}^{\infty} \text{ } \frac{1}{2}  \quad = \quad \infty, \
	\end{eqnarray*}
	as required.
	
	Finally, we endeavour to show (\ref{eqn:lemma1plan}) holds for our choice of divergent
	function.	To that end, for each non-zero integer $q$ let
	\begin{equation*}
		\textbf{R}_{\psi_o}(q):= \left\{ \boldsymbol\gamma \in [0, 1]^2:
		\max\left\{ \norm{qx_1-\gamma_1}^{1/i}, \norm{qx_2-\gamma_2}^{1/j} \right\} \leq \psi_0(\infabs{q}) \right\}
	\end{equation*}
	denote the closed rectangular region in the plane centred at the point $q\mathbf{x}$ (mod~$1$) of sidelengths $2\psi_0^{i}(\infabs{q})$ and
	$2\psi_0^{j}(\infabs{q})$ respectively.
	When using the notation `$\textbf{R}_{\psi_o}(q)$' it will be understood that $i$, $j$ and $\mathbf{x}$ are fixed. In addition, all such closed rectangular regions
	will be referred to throughout as simply a `rectangle' and all points within any such rectangle will tacitly be modulo one.
	It follows that
	\begin{eqnarray}
		\nonumber
		\textbf{W}_{(i, j)}^{\mathbf{x}}(\psi_0)& = & \left\{ \boldsymbol\gamma \in [0,1]^2: \boldsymbol\gamma \in
		\textbf{R}_{\psi_o}(q) \, \text{ for inf.\ many } q \in \mathbb{Z}_{\neq0} \right\} \\
	 \label{eqn:newchar} & = & \left\{ \right. \boldsymbol\gamma \in [0,1]^2: 
	 \boldsymbol\gamma \in \bigcup_{\infabs{q} = n_{k-1} + 1}^{n_k}
		\textbf{R}_{\psi_o}(q) \, \text{ for inf.\ many } k \in \mathbb{N} \left. \right\}.\
	\end{eqnarray}
	In view of the Borel-Cantelli lemma, to show that equation (\ref{eqn:lemma1plan}) 
	holds it is enough to show that
	\begin{equation}
		\label{eqn:newaim}
		\displaystyle\sum_{k=1}^{\infty} \mu \left( \displaystyle\bigcup_{\infabs{q} = n_{k-1} + 1}^{n_k}
		\textbf{R}_{\psi_o}(q) \right) \, < \, \infty.
	\end{equation}
	We will estimate the LHS by estimating the measure of each union of rectangles of the form
	\begin{equation*}
		\textbf{R}_{\psi_o}^{\ast}(k):\, =\, \displaystyle\bigcup_{\infabs{q} = n_{k-1} + 1}^{n_k} \textbf{R}_{\psi_o}(q),
		\quad \quad \text{ for } k \in \mathbb{N}.
	\end{equation*}
	We will hereafter refer to any union of rectangles as a `collection'.
	For each $k$, the collection $\textbf{R}_{\psi_o}^{\ast}(k)$ consists of $2(n_k-n_{k-1})$ rectangles
	in $[0,1]^2$ each centred at some point $q\mathbf{x}$ for which $n_{k-1} < \infabs{q} \leq n_k$.
	By definition, every rectangle in a collection is of the same measure, in particular each has respective
	sidelengths $2\psi_0^{i}(n_k)$ and $2\psi_0^{j}(n_k)$.	
	
	To estimate the measure of $\textbf{R}_{\psi_o}^{\ast}(k)$ we will cover it with a collection of
	larger rectangles whose measure will in some sense increase at a more controllable rate than those of $\textbf{R}_{\psi_o}^{\ast}(k)$.
	This will allow us to calculate a finite upper bound for the sum (\ref{eqn:newaim}) as required.
	With these aims in mind, for each index $k$ set
		\begin{eqnarray*}
		\textbf{S}_{\psi_o}^{\ast}(k): &  = & \displaystyle\bigcup_{\infabs{q}=1}^{\infabs{q_k}} 
		\left\{ \boldsymbol\gamma	\in [0, 1]^2:
		\norm{qx_1-\gamma_1} \leq \frac{n_k}{\infabs{q_k}} \left(\frac{c_k}{\infabs{q_k}} 
		\right)^i + \psi_0^{i}(n_k)\right. \\ 
		& & \, \left. \, \quad \quad  \quad \quad  \quad \quad \text{ and } \, \displaystyle
		\norm{qx_2-\gamma_2} \leq \frac{n_k}{\infabs{q_k}} \left(\frac{c_k}{\infabs{q_k}} 
		\right)^j + \psi_0^{j}(n_k)	\right\}. \
	\end{eqnarray*}
	Each collection $\textbf{S}_{\psi_o}^{\ast}(k)$ now consists of $2\infabs{q_k}$ rectangles in $[0,1]^2$, one
	centred at each point $q\mathbf{x}$ with $1 \leq \infabs{q}<\infabs{q_k}$.
	The sidelengths of each of these
	rectangles are
	\begin{equation*}
		2\left(\frac{n_k}{\infabs{q_k}} \left(\frac{c_k}{\infabs{q_k}} \right)^i + \psi_0^{i}(n_k)\right) \quad
		\text{	and } \quad 2\left(\frac{n_k}{\infabs{q_k}} \left(\frac{c_k}{\infabs{q_k}} \right)^j 
		+ \psi_0^{j}(n_k)\right)
	\end{equation*}
	respectively. An upper bound for the Lebesgue measure of	 $\textbf{S}_{\psi_o}^{\ast}(k)$
	can be easily deduced. We have
	\begin{equation}
		\label{eqn:Hkbound}
		\mu\left( \textbf{S}_{\psi_o}^{\ast}(k) \right) \leq 2^3 \infabs{q_k} \left( \frac{n_k}{\infabs{q_k}} \left(
		\frac{c_k}{\infabs{q_k}} \right)^i + \psi_0^{i}(n_k)\right) \left( \frac{n_k}{\infabs{q_k}} \left(
		\frac{c_k}{\infabs{q_k}} \right)^j + \psi_0^{j}(n_k)\right)
	\end{equation}
	for every index $k\geq1$.	
	
	We wish to show that $\textbf{S}_{\psi_o}^{\ast}(k)$ covers $\textbf{R}_{\psi_o}^{\ast}(k)$
	for each $k$. As the rectangles
	of $\textbf{S}_{\psi_o}^{\ast}(k)$ are larger than those of $\textbf{R}_{\psi_o}^{\ast}(k)$,
	any rectangle of $\textbf{R}_{\psi_o}^{\ast}(k)$
	centred at a point $q'\mathbf{x}$ with $n_{k-1} < \infabs{q'} \leq \infabs{q_k}$ will automatically be
	contained in the corresponding rectangle of $\textbf{S}_{\psi_o}^{\ast}(k)$. Hence, it will suffice to check that any
	rectangle of $\textbf{R}_{\psi_o}^{\ast}(k)$ centred at a point $q'\mathbf{x}$ with
	$\infabs{q_k} < \infabs{q'}
	\leq n_k$ is covered by some rectangle of $\textbf{S}_{\psi_o}^{\ast}(k)$. It is clear by construction and
  inequality	(\ref{eqn:ckproperty}) that $\infabs{q_k} < n_k$ and so rectangles of this type are present in every
  $\textbf{R}_{\psi_o}^{\ast}(k)$. For each of these integers $q'$ we can find a natural number
	$m$ such that	$\infabs{q'-m q_k} \leq \infabs{q_k}$. This implies there must be a rectangles of the collection
	$\textbf{S}_{\psi_o}^{\ast}(k)$ that is centred at the point
	$(q'-m q_k) \mathbf{x}$.
	It is also clear that $m$ can always be chosen in a way such that $\infabs{m q_k} < \infabs{q'}$.
	It follows that
	\begin{equation}
		\label{eqn:mbound}
		\infabs{m} < \frac{\infabs{q'}}{\infabs{q_k}} \leq \frac{n_k}{\infabs{q_k}}.
	\end{equation}
	Now, consider the distance between the points $q' \mathbf{x}$ and	$(q'-m q_k) \mathbf{x}$.
	We have
	\begin{eqnarray*}
		\norm{q' x_1-(q'-m q_k)x_1} & = \norm{-m q_k x_1} & \leq \infabs{m}\norm{q_k x_1} \\
		& & \stackrel{(\ref{eqn:badsequence})}{<} \infabs{m} \left( \frac{c_k}{\infabs{q_k}} \right)^i \\
		& & \stackrel{(\ref{eqn:mbound})}{<} \frac{n_k}{\infabs{q_k}} \left( \frac{c_k}{\infabs{q_k}} \right)^i, \
	\end{eqnarray*}
	and similarly
	\begin{equation*}
		\norm{q' x_2-(q'-m q_k)x_2} < \frac{n_k}{\infabs{q_k}} \left( \frac{c_k}{\infabs{q_k}} \right)^j.	
	\end{equation*}
	Combining the two above inequalities yields that any rectangle of $\textbf{R}_{\psi_o}^{\ast}(k)$ centred at a
	point $q'\mathbf{x}$ with $\infabs{q_k} < \infabs{q'} \leq n_k$ is covered by the
	rectangle of $\textbf{S}_{\psi_o}^{\ast}(k)$ centred at $(q'-m q_k) \mathbf{x}$.
	This shows that
	$\textbf{S}_{\psi_o}^{\ast}(k)$ is a cover for $\textbf{R}_{\psi_o}^{\ast}(k)$ and so
	\begin{equation*}
		\sum_{k=1}^{\infty} \mu \left( \textbf{R}_{\psi_o}^{\ast}(k) \right) \quad \leq \quad
		\sum_{k=1}^{\infty} \mu \left( \textbf{S}_{\psi_o}^{\ast}(k) \right).
	\end{equation*}
	Estimate (\ref{eqn:Hkbound}) yeilds that the RHS is bounded above by	
	\begin{eqnarray*}
	  & & \displaystyle\sum_{k=1}^{\infty} 8 \infabs{q_k}
		\left( \frac{n_k}{\infabs{q_k}} \left(
		\frac{c_k}{\infabs{q_k}} \right)^i + \psi_0^{i}(n_k)\right) \left( \frac{n_k}{\infabs{q_k}} \left(
		\frac{c_k}{\infabs{q_k}} \right)^j + \psi_0^{j}(n_k)\right) \\
		& = &	\displaystyle\sum_{k=1}^{\infty} 8 \infabs{q_{k}}
		\left( \left( c_{k} \right)^{-\frac{1}{3}} \left( c_{k} \right)^{i}\infabs{q_{k}}^{-i}
		+ \infabs{q_{k}}^{-i} \left( c_{k} \right)^{\frac{i}{3}}\right)	\\
		& & \quad \quad \quad  \times \left( \left( c_{k} \right)^{-\frac{1}{3}}
		\left( c_{k} \right)^{j}\infabs{q_{k}}^{-j} +
		\infabs{q_{k}}^{-j} \left( c_{k} \right)^{\frac{j}{3}}\right) \\
		& = & 8 \displaystyle\sum_{k=1}^{\infty} \infabs{q_{k}} \infabs{q_{k}}^{-i-j} \left( 	\left( c_{k}
		\right)^{i-\frac{1}{3}} + \left(c_{k}\right)^{\frac{i}{3}}\right)
		\left(\left( c_{k} \right)^{j-\frac{1}{3}} + \left( c_{k} \right)^{\frac{j}{3}}\right). \
	\end{eqnarray*}
	However, we have that $i+j=1$ and so this reduces to
	\begin{eqnarray*}
	  &  & 8 \displaystyle\sum_{k=1}^{\infty} \left( \left( c_{k} \right)^{i+j-\frac{2}{3}} + \left( c_{k}
	  \right)^{\frac{i+j}{3}} +	\left( c_{k} \right)^{\frac{i}{3}+j - \frac{1}{3}} +
		\left( c_{k} \right)^{i+\frac{j}{3}-\frac{1}{3}} \right) \\
		& = & 8 \displaystyle\sum_{k=1}^{\infty} \left( 2 \left( c_{k} \right)^{\frac{1}{3}}
		+	\left( c_{k} \right)^{\frac{2i}{3}} + \left( c_{k} \right)^{\frac{2j}{3}} \right) \\
		& \leq & 8 \displaystyle\sum_{k=1}^{\infty} 4 \left( c_{k} \right)^{2 \min\left\{i, j \right\}/3} \\
		& \stackrel{(\ref{eqn:ckcondition})}{<} & 32 \displaystyle\sum_{k=1}^{\infty} \left( c_{1}
		\right)^{2 \min\left\{i, j \right\}/3}  2^{-(k-1)}  \\
		& = & 64 \left( c_{1} \right)^{2 \min\left\{i, j \right\}/3} \quad < \quad \infty, \
	\end{eqnarray*}
	\noindent
	as required. This completes the proof of Lemma \ref{prop:kurz1}.
\end{proof}

\subsection{Proof of Theorem~\ref{thm:kurz} (Part $2$)}
\label{sec:proof-of-theorem-kurz-part-2}

In this section we prove the complementary inclusion to that of Proposition \ref{prop:kurz1}.

\begin{prop}
	\label{prop:kurz2}
	For every real $i,j > 0$ such that $i+j=1$,
	\begin{equation*}
		\label{eqn:kurz2}
		\Bad(i, j) \subseteq \bigcap_{\psi \in \mathcal{D}} \textbf{V}_{(i,j)}(\psi).	
	\end{equation*}
\end{prop}

\begin{proof}

	We are required to show that if $\mathbf{x} \in \Bad(i, j)$ then for every divergent approximating function $\psi$
	we have that
	\begin{equation*}
		\mu \left( \textbf{W}_{(i, j)}^{\mathbf{x}}(\psi) \right) =1.
	\end{equation*}
	To do this we first prove the intermediary result that for every $\mathbf{x} \in \Bad(i, j)$ we have
	\begin{equation}
		\label{eqn:posmeas}
		\mu \left( \textbf{W}_{(i, j)}^{\mathbf{x}}(\psi) \right) >0
	\end{equation}
	for every $\psi \in \mathcal{D}$.
	
	Fix $\mathbf{x} \in \Bad(i, j)$. By definition there exists a constant $c(\mathbf{x})>0$ such that
	for all natural numbers $q$
	\begin{equation*}
		\max\left\{ \norm{q x_1}^{1/i}, \norm{q x_2}^{1/j} \right\} > \frac{c(\mathbf{x})}{q}.
	\end{equation*}
	Next, choose any function $\psi \in \mathcal{D}$. To ensure that certain technical
	conditions required later in the proof are met we will work with a refinement of $\psi$. Let
	\begin{equation*}
		a^{\ast}:= 2^{-1/ \max\left\{i, j\right\}} \quad \quad \text{ and } \quad \quad
		a_{\ast}:= 2^{-1/ \min\left\{i, j \right\}},
	\end{equation*}
	then for each $r \in \mathbb{N}$ set		
	\begin{equation*}
		\psi_1(r): = \min \left\{ \psi(r), \quad \frac{a^{\ast}}{2}, \quad  
		\frac{a_{\ast}\, c(\mathbf{x})}{2 \infabs{r}}  \right\}.
	\end{equation*}
	Finally, choose any integer $k$	such that
	\begin{equation}
		\label{eqn:kbound}
		k> 4
	\end{equation}
	and for each natural number $r$ define
	\begin{equation*}
		\psi_2(r): =
		\begin{cases}
			\psi_1(k), & \quad r \leq k. \\
  		\psi_1(k^{t+1}), & \quad k^t < r \leq k^{t+1} \quad \quad \quad \text{ for each } t \in \mathbb{N}.
  	\end{cases}
	\end{equation*}
	
	It is easy to see that for each $r \in \mathbb{N}$
	\begin{equation}
		\label{eqn:seqorder}
		\psi_2(r) \leq \psi_1(r) \leq \psi(r)
	\end{equation}
	and that $\psi_1 \in \mathcal{D}$.
	It is also clear that $\psi_2$ is decreasing and strictly positive.  Furthermore,
	\begin{eqnarray*}
		\sum_{r=1}^{\infty} \psi_2(r) & \geq & \sum_{t=1}^{\infty} \, \sum_{r=k^t+1}^{k^{t+1}} \psi_2(r) \\
		& =
		& \sum_{t=1}^{\infty} \left( k^{t+1} -  k^t \right) \psi_2(k^{t+1}) \\
		& =
		& \frac{1}{k} \sum_{t=1}^{\infty} \left( k^{t+2} - k^{t+1} \right)
		\psi_1(k^{t+1}) \\
		& \geq
		& \frac{1}{k} \sum_{t=1}^{\infty} \, \sum_{r=k^{t+1}+1}^{k^{t+2}}
		\psi_1(r) \\
		& = & \frac{1}{k} \, \sum_{r=k^2+1}^{\infty} \psi_1(r) \quad = \quad \infty, \
	\end{eqnarray*}
	and so $\psi_2$ too is a divergent approximating function.
	
	With reference to \S\ref{sec:proof-of-theorem-kurz-part-1}, inequality (\ref{eqn:seqorder}) and the characterisation
  of $\textbf{W}_{(i, j)}^{\mathbf{x}}(\psi)$ in terms of the rectangles $\textbf{R}_{\psi}(q)$ given by (\ref{eqn:newchar}) now guarantee
	that the following statement is sufficient to prove that
	(\ref{eqn:posmeas}) holds for every choice of function $\psi$.
	For every integer $r \geq 1$
	\begin{equation}
		\label{eqn:posmeas'}
		\mu \left( \bigcup_{\infabs{q}=r+1}^{\infty} \textbf{R}_{\psi_2}(q) \right) \,  \geq \, 
		a_\ast \, c(\mathbf{x})/8.
	\end{equation}
	Note that this statement is in terms of the constructed function $\psi_2$.
	To prove (\ref{eqn:posmeas'}) we will show that there cannot exist a natural number $t_0$ such that it fails to hold when $r=k^{t_0}$.
	Assume that such a $t_0$ exists and consider the collection of rectangles defined by
	\begin{equation*}
		\textbf{R}_t:= \textbf{R}\left( \psi_2, t \right):= \bigcup_{\infabs{q}=k^{t_o}+1}^{k^t} \textbf{R}_{\psi_2}(q)
		\quad \quad \quad \text{ for } t=t_0+1, \text{ } t_0+2, \ldots.
	\end{equation*}
	We will demonstrate that the measure of the set $\textbf{R}_t$ is unbounded as $t$ increases and in doing so reach a contradiction, 
	as each $\textbf{R}_t$ is contained in $[0,1]^2$. We will do this by estimating the size of a suitable sum of the measure of 
	set differences of the form $\textbf{R}_{t+1} \setminus \textbf{R}_t$.
	
	By construction each $\textbf{R}_{t+1}$ is obtained from $\textbf{R}_t$ by adding $2(k^{t+1}-k^t)$ new rectangles
	to those of $\textbf{R}_t$. These new rectangles are centred at the points $q\mathbf{x}$
	for which $k^t < \infabs{q} \leq k^{t+1}$.
	To estimate $\mu \left( \textbf{R}_{t+1} \setminus \textbf{R}_t \right)$
	we will find an upper bound to the number of the new rectangles that intersect any existing rectangle of
	$\textbf{R}_t$.  In particular, we will find an upper bound to the cardinality
	of the set $\textbf{J}_{t+1} \cap \textbf{2R}_t$, where $\textbf{J}_{t+1}$ denotes the set of points
	$q\mathbf{x}$ for which $k^t < \infabs{q} \leq k^{t+1}$ and
	\begin{equation*}
		\textbf{2R}_t:= \bigcup_{\infabs{q}=k^{t_o}+1}^{k^t} \textbf{R}_{2\psi_2}(q)  \quad \quad \quad
		\text{ for } t=t_0+1, t_0+2, \ldots.
	\end{equation*}
	This will suffice as $\psi_2$ is non-increasing. Before proceeding we first notice that, since the vector $\mathbf{x}$ was chosen from $\Bad(i,j)$,	if $q\mathbf{x}$ and  $q'\mathbf{x}$ are members of $\textbf{J}_{t+1}$ then
	\begin{equation}
		\label{eqn:jseparation}
		\max \left\{ \norm{qx_1-q'x_1}^{1/i}, \norm{qx_2-q'x_2}^{1/j} \right\} \, \geq
		 \,	\frac{c(\mathbf{x})}{\infabs{q-q'}} \,
		 \geq \, \frac{c(\mathbf{x})}{2k^{t+1}},
	\end{equation}
	providing that the integers $q$ and $q'$ are distinct.

	The collection $\textbf{2R}_t$ can be partitioned into two exhaustive subcollections (which we will assume without loss of generality
	are non-empty). Recalling that
	$a_{\ast}:= 2^{-1/ \min\left\{i, j\right\}}$, define
	\begin{equation*}
		\textbf{2R}_t^{(1)}:= \, \bigcup \, \textbf{R}_{2\psi_2}(q),
	\end{equation*}
	where the union runs over all non-zero $q$ with $k^{t_0} < \infabs{q} \leq k^t$ such that
	\begin{equation*}
		2\psi_2(\infabs{q}) \, < \, \frac{a_{\ast} \,c(\mathbf{x})}{2k^{t+1}}.
	\end{equation*}
	In turn, let
	\begin{equation*}
		\textbf{2R}_t^{(2)}:= \, \bigcup \, \textbf{R}_{2\psi_2}(q),
	\end{equation*}
	where this time the union runs over $q$ with $k^{t_0} < \infabs{q} \leq k^t$ such that
	\begin{equation*}
		2\psi_2(\infabs{q}) \,  \geq \, \frac{a_{\ast} \,c(\mathbf{x})}{2k^{t+1}}.
	\end{equation*}
	The intersections $\textbf{J}_{t+1} \cap \textbf{2R}_t^{(1)}$ and $\textbf{J}_{t+1} \cap \textbf{2R}_t^{(2)}$
	will now be dealt with independently.
	
	The subcollection $\textbf{2R}_t^{(1)}$ consists of rectangles of sidelengths
	\begin{equation*}
	2(2\psi_2(\infabs{q}))^{i} \quad \text{ and } \quad 2(2\psi_2(\infabs{q}))^{j}
	\end{equation*}
	respectively and both
	\begin{equation*}
		2 \left( 2\psi_2(\infabs{q}) \right)^{i} \, < \, \left(\frac{c(\mathbf{x})}{2k^{t+1}} \right)^i
		\quad \quad \text{ and } \quad \quad
		2 \left( 2\psi_2(\infabs{q}) \right)^{j} \, < \, \left(\frac{c(\mathbf{x})}{2k^{t+1}} \right)^j.
	\end{equation*}
	This follows upon noticing that $\max \left\{a_\ast^i, a_\ast^j \right\}=1/2$.
	Thus, statement (\ref{eqn:jseparation}) implies at most one element of $\textbf{J}_{t+1}$ can lie in each
	rectangle of $\textbf{2R}_t^{(1)}$ and so $\textbf{J}_{t+1} \cap \textbf{2R}_t^{(1)}$ contains at most
	$2(k^t-k^{t_0}) < 2k^t$ elements.
	
	Estimating the cardinality of $\textbf{J}_{t+1} \cap \textbf{2R}_t^{(2)}$ requires more work and we 
	argue as follows. 	
	If a point $\boldsymbol\gamma_0$ lies in the subcollection $\textbf{2R}_t^{(2)}$ then it must lie in a rectangle of the form
	$\textbf{R}_{2\psi_2}(q_0) \subseteq \textbf{2R}_t^{(2)}$
	for some integer $q_0$ with $k^{t_0} < \infabs{q_0} \leq k^t$. This rectangle must have respective sidelengths
	$2(2\psi_2(\infabs{q_0}))^{i}$ and $2(2\psi_2(\infabs{q_0}))^{j}$ and by definition we have
	\begin{equation*}
		2 \left( 2\psi_2(\infabs{q_0}) \right)^{i} \, \geq \,  2\left(\frac{a_\ast \, 
		c(\mathbf{x})}{2k^{t+1}} \right)^i
		\quad \, \text{ and } \quad \,
		2 \left( 2\psi_2(\infabs{q_0}) \right)^{j} \, \geq \,  2 \left(\frac{a_\ast \, 
		c(\mathbf{x})}{2k^{t+1}} \right)^j.
	\end{equation*}
	It is now clear that there must exist a point $\mathbf{y}(\boldsymbol\gamma_0) \in \textbf{R}_{2\psi_2}(q_0)$
	such that $\boldsymbol\gamma_0 $ is contained in a subrectangle, say $\textbf{S}(\boldsymbol\gamma_0)$, of
	$\textbf{R}_{2\psi_2}(q_0)$ centred at $\mathbf{y}(\boldsymbol\gamma_0)$ and of sidelengths
	$\left(a_\ast \,c(\mathbf{x})/2k^{t+1} \right)^i$ and
	$ \left( a_\ast \,c(\mathbf{x})/2k^{t+1} \right)^j$ respectively.
	The fact that $\max \left\{a_\ast^i, a_\ast^j \right\}=1/2$, twinned with equation
	(\ref{eqn:jseparation}), once more guarantees that only one point of $\textbf{J}_{t+1}$ may lie in any
	subrectangle of this type.
	Moreover, any two such subrectangles containing respective points $q\mathbf{x}$ and $q'\mathbf{x}$, both in
	$\textbf{J}_{t+1}$, must be disjoint.  Thus, the cardinality of
	$\textbf{J}_{t+1} \cap \textbf{2R}_t^{(2)}$ cannot exceed $\mu ( \textbf{2R}_t^{(2)} ) / \mu
	\left(\textbf{S}(\boldsymbol\gamma_0)\right)$.  We estimate the size of $\mu ( \textbf{2R}_t^{(2)} )$ by
	utilising the following lemma.	
	\begin{lem}
		\label{lem:hrbound}
		For every $t=t_0+1, t_0+2, \ldots$,
		\begin{equation*}
			\mu \left( \textbf{2R}_t \right) \, \leq \, 2 \mu \left( \textbf{R}_t \right).
		\end{equation*}		
	\end{lem}
	\begin{proof}[Proof of Lemma \ref{lem:hrbound}]
		For $s \in \mathbb{N}$, let
		\begin{equation*}
			\textbf{R}^s:= \bigcup_{\infabs{q}=k^{t_o}+1}^{k^{t_o}+s} \textbf{R}_{\psi_2}(q) \quad \quad \text{ and } \quad
			\quad \textbf{2R}^s:= \bigcup_{\infabs{q}=k^{t_o}+1}^{k^{t_o}+s} \textbf{R}_{2\psi_2}(q).
		\end{equation*}
		To prove Lemma \ref{lem:hrbound} it suffices to show that
		$\mu \left( \textbf{2R}^s \right) \leq 2 \mu \left( \textbf{R}^s \right)$	for all $s$.
		We proceed by induction. If $s=1$, then
		\begin{equation*}
			\mu (\textbf{R}^1) \, = \, 2\psi_2^{i}(k^{t_0}+1) \cdot 2\psi_2^{j}(k^{t_0}+1) \, = \, 
			4\psi_2(k^{t_0}+1).
		\end{equation*}
		Further,
		\begin{equation*}
			\mu ( \textbf{2R}^1 ) \, = \, 2(2\psi_2(k^{t_0}+1))^{i} \cdot 2(2\psi_2(k^{t_0}+1))^{j} \, = \, 2\cdot
			4\psi_2(k^{t_0}+1) \, = \, 2 \mu (\textbf{R}^1)
		\end{equation*}
		and the statement holds.
		
		Next, assume the hypothesis holds when $s=s'$ and define a transformation $T$ on the torus $[0,1]^2$ by
		\begin{equation*}
			T(\boldsymbol\gamma):= \left( 2^{i}\gamma_1, \text{ } 2^{j}\gamma_2 \right) \quad \quad \quad
			\forall \text{ } \boldsymbol\gamma \in [0,1]^2.
		\end{equation*}
		For any subset $\textbf{D} \subseteq [0,1]^2$, we denote by $T(\textbf{D})$ the set of all points
		$T(\boldsymbol\gamma)$ where $\boldsymbol\gamma \in \textbf{D}$.
		Let $\textbf{D}^{s'+1}:= \textbf{R}^{s'+1} \setminus \textbf{R}^{s'}$, then, since by definition 
		$\psi_2$ does not exceed $a^{\ast}(i,j) / 2$,
		\begin{equation}
			\label{eqn:lemma1}
			\mu( T(\textbf{D}^{s'+1})) = 2^{i} \cdot 2^{j} \cdot \mu(\textbf{D}^{s'+1}) =
			2\mu(\textbf{D}^{s'+1}).
		\end{equation}
		It is also clear that
		\begin{equation*}
			\textbf{2R}^{s'+1} = \textbf{2R}^{s'} \cup T(\textbf{D}^{s'+1}),
		\end{equation*}
		from which it follows that
		\begin{eqnarray*}
			\mu(\textbf{2R}^{s'+1}) & = & \mu( \textbf{2R}^{s'} \cup T(\textbf{D}^{s'+1})) \\
			& \leq & \mu( \textbf{2R}^{s'}) + \mu( T(\textbf{D}^{s'+1})) \\
			& \leq & 2 \mu(\textbf{R}^{s'}) + 2\mu(\textbf{D}^{s'+1}) \quad \quad
			\text{ (by assumption and (\ref{eqn:lemma1}) resp.)} \\
			& = & 2 \mu(\textbf{R}^{s'} \cup \textbf{D}^{s'+1}) \quad \quad \quad \quad
			\text{ (since } \textbf{R}^{s'} \text{ and } \textbf{D}^{s'+1} \text{ are disjoint)} \\
			& = & 2\mu(\textbf{R}^{s'+1}), \
		\end{eqnarray*}
		as required.
	\end{proof}
	
	We return to our calculation. 
	Assuming as we are that statement (\ref{eqn:posmeas'}) is	false, Lemma~\ref{lem:hrbound} now yields that
	\begin{equation*}
		\mu(\textbf{2R}_t^{(2)}) \, \leq \, \mu(\textbf{2R}_t) \, \leq \, 2\mu(\textbf{R}_t) \, 
		< \, a_\ast \, c(\mathbf{x})/4.
	\end{equation*}
	Thus,
	\begin{equation*}
		\# ( \textbf{J}_{t+1} \cap \textbf{2R}_t^{(2)} ) \quad \leq \quad
		\frac{\mu ( \textbf{2R}_t^{(2)})}{\mu(\textbf{S}(\boldsymbol\gamma_0))} \quad
		< \quad \frac{a_\ast \, c(\mathbf{x})}{ 4\left(
		a_\ast \, c(\mathbf{x})/2k^{t+1} \right)^{i + j}} \quad = \quad \frac{k^{t+1}}{2}
	\end{equation*}
	and we have found our second upper bound. 
	
	 Recalling our intention to estimate $\mu \left( \textbf{R}_{t+1} \setminus \textbf{R}_t \right)$, we 
	 can now write down an upper bound for the number of rectangles added to $\textbf{R}_t$ to make 
	 $\textbf{R}_{t+1}$ that do intersect existing rectangles of $\textbf{R}_t$. 
	Indeed, this number cannot exceed
	\begin{equation}	
		\label{eqn:cardJt}
		 \# \left( \textbf{J}_{t+1} \cap \textbf{2R}_t \right) \quad \leq \quad 2k^t + k^{t+1}/2,
	\end{equation}
	which follows upon noticing that
	\begin{equation*}
		\textbf{J}_{t+1} \cap \textbf{2R}_t \quad = \quad (\textbf{J}_{t+1} \cap \textbf{2R}_t^{(1)}) \cup
		(\textbf{J}_{t+1} \cap \textbf{2R}_t^{(2)}).
	\end{equation*}
	To complete our arguement we require one final piece of notation. Let
	\begin{equation*}
		\textbf{L}_{t+1}:=\left\{q \in \mathbb{Z}_{\neq0}: q\mathbf{x} \in \textbf{J}_{t+1},
		q\mathbf{x} \notin \textbf{2R}_t \right\}.
	\end{equation*}
	The integers $q \in \textbf{L}_{t+1}$ each correspond to a rectangle of
	$\textbf{R}_{t+1}$ that does not intersect any rectangle of $\textbf{R}_t$.
	So, by (\ref{eqn:cardJt})
	\begin{eqnarray}
		 \nonumber \# (\textbf{L}_{t+1}) & \geq & 2(k^{t+1}-k^t) - (2k^t + k^{t+1}/2) \\
		 \nonumber & = & (2-4/k-1/2) k^{t+1} \\
		 \nonumber & \stackrel{(\ref{eqn:kbound})}{>} & (2-1-1/2) k^{t+1} \\
		 \label{eqn:boundinarray} & = & k^{t+1}/2.\
	\end{eqnarray}
	
	We will now estimate $\mu\left( \textbf{R}_{t+1} \setminus \textbf{R}_t \right)$ by considering the inclusion
	\begin{equation}
		\label{eqn:differenceinclusion}
		\textbf{R}_{t+1} \setminus \textbf{R}_t \quad \supset \quad \bigcup_{q \in \textbf{L}_{t+1}} \textbf{R}_{\psi_2}(q).
	\end{equation}
	The rectangles $\textbf{R}_{\psi_2}(q)$ in the above union have sidelengths $2\psi_2^{i}(\infabs{q})$ and $2\psi_2^{j}(\infabs{q})$ respectively.
	Further, if $q, q' \in \textbf{L}_{t+1}$ then $k^t<\infabs{q}, \infabs{q'} \leq k^{t+1}$
	and so
	\begin{equation}
		\label{eqn:badlt}
		\max \left\{ \norm{qx_1-q'x_1}^{1/i}, \norm{qx_2-q'x_2}^{1/j} \right\} \quad \stackrel{(\ref{eqn:jseparation})}{\geq} \quad
		\frac{c(\mathbf{x})}{2k^{t+1}}.
	\end{equation}
	Recall that $\psi_2$ is constant on each $\textbf{L}_{t+1}$ by definition, taking the value 
	$\psi_2(k^{t+1})$, and also that
	\begin{equation*}
		 \psi_2(r) \, \leq \,  \frac{a_{\ast}\, c(\mathbf{x})}{2 \infabs{r}}.
	\end{equation*}
	Therefore, we have both
	\begin{equation*}
		2 \psi_2^{i}(\infabs{q})  \, = \, 2\psi_2^{i}(k^{t+1}) \, 
		< \, \left(\frac{c(\mathbf{x})}{2k^{t+1}} \right)^i
	\end{equation*}	
	and 
	\begin{equation*}
		2 \psi_2^{j}(\infabs{q})  \, = \, 2\psi_2^{j}(k^{t+1}) \, 
		< \, \left(\frac{c(\mathbf{x})}{2k^{t+1}} \right)^j.
	\end{equation*}	
	Combining these inequalities with statement (\ref{eqn:badlt}) yields that the rectangles $\textbf{R}_{\psi_2}(q)$ on the RHS of
	(\ref{eqn:differenceinclusion})	are disjoint. Hence,	
	\begin{eqnarray*}
		\mu \left( \textbf{R}_{t+1} \setminus \textbf{R}_t \right) & \geq & \sum_{q \in \textbf{L}_{t+1}} 
		\mu\left( R_{\psi_2}(q) \right) \\
		& = & 2^2 \sum_{q \in \textbf{L}_{t+1}} \psi_2(\infabs{q}) \\
		& \stackrel{(\ref{eqn:boundinarray})}{>}
		& 2k^{t+1} \psi_2(k^{t+1}) \\
		& > & 2(k^{t+1}-k^t)\psi_1(k^{t+1}) \\ 
		& = & \sum_{\infabs{q}=k^t+1}^{k^{t+1}}	\psi_1(k^{t+1})\\ 
		& = & \sum_{\infabs{q}=k^t+1}^{k^{t+1}} \psi_1(\infabs{q}). \
	\end{eqnarray*}
	Finally, $\psi_1$ is divergent; i.e.
	\begin{equation*}
		\sum_{\infabs{q}=1}^{\infty} \psi_1(\infabs{q})  \quad = \quad \infty,
	\end{equation*}
	whence $\sum_{t > t_0}\mu \left( \textbf{R}_{t+1} \setminus \textbf{R}_t \right) = \infty$.
	Since $ R_t \subseteq R_{t+1}$ for any $t > t_0$, this implies that 
	$\mu(R_t) \rightarrow \infty$ as $t\rightarrow \infty$. 
	However, each set $R_t$ is contained in $[0,1]^2$ and so a contradiction is reached. This means 
	the assumption that (\ref{eqn:posmeas'}) fails for some $r=k^{t_o}$ 
	is indeed false, and consequently 
	\begin{equation*}
		\mu \left( \textbf{W}_{(i, j)}^{\mathbf{x}}(\psi) \right) >0
	\end{equation*}
	for every $\psi \in \mathcal{D}$ as desired.

	To complete the proof of Proposition \ref{prop:kurz2} we must now show if $\mathbf{x} \in \Bad(i, j)$ then
	\begin{equation*}
		\mu \left( \textbf{W}_{(i, j)}^{\mathbf{x}}(\psi) \right) =1
	\end{equation*}
	for every $\psi \in \mathcal{D}$.
	Our method will be through the application of two lemmas, the first of which is due to Kurzweil (\cite[Lemma 13]{MR0073654}).
	\begin{lem}[Kurzweil]
		\label{lem:density}
		Let $U$ and $V$ be subsets of $[0,1]^2$. If $\mu(U)>0$ and $V$ is dense in $[0,1]^2$ then
		$\mu(U \oplus V) = 1$, where $U \oplus V:=\left\{u+v \right. \, ($mod~$ 1):\left. \, u \in U, v \in V \right\}$.
	\end{lem}
	\noindent
	\begin{lem}
		\label{lem:subseries}	
		For every $\psi \in	\mathcal{D}$ and for every natural number $s$ we have
		\begin{equation*}
			\sum_{r=1}^{\infty} \psi(sr) = \infty.
		\end{equation*}
	\end{lem}
	\begin{proof}[Proof of Lemma \ref{lem:subseries}]
		Suppose $s\geq1$ and for ease of notation set $\psi(0):=\psi(1)$. Consider the $s$-subseries
		$\sum_{r=0}^\infty \psi(sr+k)$ for each $k=0,\dots,s-1$. Every term $\psi(r')$, $r' \in \mathbb{N}$,
		appears exactly once in exactly one $s$-subseries.
		If every $s$-subseries had a finite sum then the original series $\sum_{r=1}^{\infty} \psi(r)$ would
		also have a finite sum (precisely equal to the sum of the sums of the $s$-subseries).
		Since the original series does not have a finite sum, at least one of the $s$-subseries must diverge,
		say $\sum_{r=0}^\infty \psi(sr+k_0)=\infty$. Since $\psi$ is decreasing
		$\psi(sr)\geq \psi(sr+k_0)$ and so $\sum_{r=0}^\infty \psi(sr)=\infty$ and Lemma
		\ref{lem:subseries} holds.
	\end{proof}
	
	Returning to the proof of Proposition \ref{prop:kurz2}, fix a divergent approximating function $\psi$ and a vector $\mathbf{x} \in \Bad(i, j)$.
	Once again, we will refine $\psi$ before proceeding. Firstly, we will construct a function $\psi_3 \in \mathcal{D}$ such that
	\begin{equation}
		\label{eqn:seqlim}
		\lim_{r\rightarrow \infty} \left( \frac{\psi_3(r)}{\psi(r)} \right) = 0.
	\end{equation}
	Let $r_0=0$ and choose $r_1 \geq 1$ such that the inequality $\sum_{r=1}^{r_1} \psi(r) \geq 1$ holds.
	Then in general	construct inductively a strictly increasing sequence $\left\{r_k\right\}_{k=0}^\infty$
	such that for each $k$
	\begin{equation}
		\label{eqn:seqdef}
		\sum_{r=r_{k-1}+1}^{r_k}\psi(r) \geq k.
	\end{equation}
	This is always possible since $\sum_{r=1}^{\infty}\psi(r)$ diverges, so the partial 
	sums from any starting point must
	tend to infinity.
	Next, define $c_r:=1/\sqrt{k}$ if $r_{k-1}<r\leq r_k$ and $\psi_3(r):=c_r\psi(r)$.
	Equation (\ref{eqn:seqlim}) therefore holds as $\psi_3(r)/\psi(r)=c_r$ tends to zero.
	Both $\psi$ and $\left\{c_r\right\}$ are strictly positive and decreasing,
	hence $\psi_3$ is strictly positive and decreasing.  Also, by construction, inequality
	(\ref{eqn:seqdef}) guarantees that
	\begin{equation*}
		\sum_{r=r_{k-1}+1}^{r_k}\psi_3(r) \, = \, \frac{1}{k}\sum_{r=r_{k-1}+1}^{r_k}\psi(r) \,
		\geq \, 1,
	\end{equation*}
	and so
	\begin{equation*}
		\sum_{r=1}^{r_k} \psi_3(r)\geq k.
	\end{equation*}
	This shows that the sum of $\psi_3$ diverges and we have verified that $\psi_3 \in \mathcal{D}$.
	
	By Lemma \ref{lem:subseries},
	\begin{equation*}
		\sum_{r=1}^{\infty} \psi_3(sr) = \infty,
	\end{equation*}
	for every natural number $s$.
	Consequently, there must exist a strictly increasing sequence of natural numbers $\left\{s_r\right\}_{r \in \mathbb{N}}$
	with $s_r \rightarrow \infty$ as $r \rightarrow \infty$ such that
	\begin{equation*}
		\sum_{r=1}^{\infty}  \psi_3(s_r \cdot r) = \infty.
	\end{equation*}
	Accordingly, we define $\psi_4(r):=\psi_3(s_r \cdot r)$. Hence, for any fixed non-zero integer $q'$ we have that
	\begin{equation}
		\label{eqn:seq3lim}
		\lim_{\infabs{q} \rightarrow \infty} \left( \frac{\psi_4(\infabs{q})}{\psi(\infabs{q+q'})} \right) = 0.
	\end{equation}
	It is also clear that $\psi_4$ is a divergent approximating function and therefore we know by intermediary result (\ref{eqn:posmeas})
	that
	\begin{equation}
		\label{eqn:Umeasure}
		\mu \left( \textbf{W}_{(i, j)}^{\mathbf{x}}(\psi_4)  \right) \, > \, 0.
	\end{equation}
	In addition, if we choose some vector $\mathbf{y}$ such that
	\begin{equation*}
		\mathbf{y} \, \in \, \textbf{W}_{(i, j)}^{\mathbf{x}}(\psi_4) \, \, \stackrel{(\ref{eqn:newchar})}{=} \, \,
		\bigcap_{k=1}^{\infty} \bigcup_{\infabs{q}=k}^{\infty} \textbf{R}_{\psi_4}(q),
	\end{equation*}
	then for every natural number $k$ there are infinitely many integers $q$ with $\infabs{q} \geq k$ such that
	$\mathbf{y} \in \textbf{R}_{\psi_4}(q)$.	It follows that $\mathbf{y}+q'\mathbf{x}$ is a member of the set 
	of $\boldsymbol\gamma \in [0, 1]^2$ for which
	\begin{equation*}
		\max\left\{ \norm{(q+q')x_1-\gamma_1}^{1/i}, \,  \norm{(q+q')x_2-\gamma_2}^{1/j} \right\}
		\, \, \leq \, \, \psi_4(\infabs{q})
	\end{equation*}
	for infinitely many integers $q$ satisfying $\infabs{q} \geq k$.
	For large enough $k$, equation (\ref{eqn:seq3lim}) implies that for each $q$ with $\infabs{q} \geq k$ the 
	set of $\boldsymbol\gamma$ defined above is contained in the rectangle
	$\textbf{R}_{\psi}(q+q')$.
	It follows that $\mathbf{y}+q'\mathbf{x}$ is contained in infinitely many rectangles of the form $\textbf{R}_{\psi}(q)$; i.e.,
	\begin{equation}
		\label{eqn:fullcontained}
		\mathbf{y}+q'\mathbf{x} \, \in \, \bigcap_{k=1}^{\infty} \bigcup_{\infabs{q}=k}^{\infty}
		\textbf{R}_{\psi}(q) \, = \, \textbf{W}_{(i, j)}^{\mathbf{x}}(\psi)
	\end{equation}
	for every natural number $q'$.	
	
	We are now in a position to apply Lemma \ref{lem:density}.
	With reference to the lemma, set
	\begin{equation*}
		U:= \textbf{W}_{(i, j)}^{\mathbf{x}}(\psi_4) \quad \text{and} \quad V:= \left\{q\mathbf{x}: q \in
		\mathbb{Z}_{\neq0} \right\}.
	\end{equation*}
	By equation (\ref{eqn:Umeasure}) we have $\mu(U)>0$ and, as mentioned in \S\ref{sec:Background2},
	Kronecker's Theorem implies that $V$ is dense in $[0,1]^2$ if $\mathbf{x}$ is irrational.
	Hence, Lemma \ref{lem:density} implies that $\mu(U\oplus V)=1$,
	from which equation (\ref{eqn:fullcontained}) gives
	\begin{equation*}
		\mu \left( \textbf{W}_{(i, j)}^{\mathbf{x}}(\psi) \right) =1
	\end{equation*}
	and the proof of Proposition \ref{prop:kurz2}, and indeed that of Theorem~\ref{thm:kurz}, is complete.
\end{proof}

\section{Proof of Theorem~\ref{thm:dimbad}}
\label{sec:ProofOfTheoremRefThmDimbad}

The proof of Theorem~\ref{thm:dimbad} makes use of the framework developed in
\cite{MR2231044}.
This framework was specifically designed to provide dimension results for a broad range of badly approximable sets.
In this section we show that $\Bad^{\mathbf{x}}(i,j)$ falls into this category when $\mathbf{x}$ is chosen
from $\Bad(i,j)$. First, we provide a simplification of the framework tailored to our needs.

Let $\mathcal{R} := \left\{ R_{\alpha} \subset \mathbb{R}^2 : \alpha \in J \right\} $ be a family of
subsets $R_{\alpha}$ of $\mathbb{R}^2$ indexed by an infinite countable set $J$. We will refer to the sets
$R_{\alpha}$ as \emph{resonant sets}. Furthermore, it will be assumed that each resonant set takes the form 
of a cartesian product; i.e., that each set $R_{\alpha}$ can be split into the images $R_{\alpha, t} \subset \mathbb{R}$, $t=1,2$, 
of its two projection maps along the two coordinate axis.
Next, let $\beta :J\rightarrow \mathbb{R}_{>0}:\alpha
\mapsto \beta_{\alpha}$ be a positive function on $J$ such that the number of
$\alpha \in J$ with $\beta_{\alpha}$ bounded above is finite. Thus, as $\alpha$ runs through $J$
the function $\beta_{\alpha}$ tends to infinity. Also, for $t=1,2$, let $\rho_t: \mathbb{R}_{>0} \rightarrow
\mathbb{R}_{>0}:r \mapsto \rho_t(r)$ be any real, positive, decreasing function such that $\rho_t(r) \rightarrow
0$ as $r \rightarrow \infty$.  We assume that either $\rho_1(r)\geq \rho_2(r)$ or $\rho_2(r)\geq \rho_1(r)$ for
large enough $r$.  Finally, for each resonant set $R_{\alpha}$ define a rectangular
neighbourhood $\mathcal{F}_{\alpha}(\rho_1, \rho_2)$ by
\begin{equation*}
	\mathcal{F}_{\alpha}(\rho_1, \rho_2):=\left\{ \mathbf{x} \in \mathbb{R}^2: \infabs{x_t - R_{\alpha, t}} \leq \rho_t
	\left( \beta_{\alpha} \right) \, \text{ for } t=1, 2 \right\},
\end{equation*}
where $\infabs{x_t - R_{\alpha, t}}:= \inf_{a \in \mathcal{R}_{\alpha, t}}\infabs{x_t - a}$.

We now introduce the general badly approximable set to which the results of \cite{MR2231044} relate. Define 
$\Bad(\mathcal{R}, \beta, \rho_1, \rho_2)$ to be the set of $\mathbf{x} \in [0,1]^2$ for which there exists 
a constant $c(\mathbf{x})>0$ such that
\begin{equation*}
  \mathbf{x} \notin  c(\mathbf{x})\mathcal{F}_{\alpha}(\rho_1, \rho_2)
  \quad \quad \forall \text {  } \alpha \in  J.
\end{equation*}
That is, $\mathbf{x} \in \Bad(\mathcal{R}, \beta, \rho_1, \rho_2)$ if there exists a constant $c(\mathbf{x})>0$
such that for all $\alpha \in J$
\begin{equation*}
	\infabs{x_t - R_{\alpha, t}} \, \, \geq \, \, c(\mathbf{x}) 
	\rho_t \left( \beta_{\alpha} \right) \quad \quad (t=1, 2).
\end{equation*}

The aim of the framework is to determine conditions under which the set
$\Bad(\mathcal{R}, \beta, \rho_1, \rho_2)$ has full Hausdorff dimension.
With this in mind, we begin with some useful notation.
For any fixed integers $k>1$ and $n \geq 1$, define
\begin{equation*}
	F_n:= \left\{ \mathbf{x} \in [0,1]^2 : \infabs{x_t - c_t} \leq \rho_t(k^n) \text{ for each } t=1,2 \right\}
\end{equation*}
to be the generic closed rectangle in $[0,1]^2$ with
centre $\mathbf{c}:=(c_1, c_2)$ and of side lengths given by $2\rho_1(k^n)$ and $2\rho_2(k^n)$ respectively.
Next, for any $\theta \in \mathbb{R}_{>0}$,
let
\begin{equation*}
\theta F_n := \left\{ \mathbf{x} \in [0,1]^2: \infabs{x_t - c_t} \leq
\theta\rho_t(k^n) \text{ for each } t=1,2 \right\}
\end{equation*}
denote the rectangle $F_n$ scaled by $\theta$.
Finally,  let $$J(n) := \left\{ \alpha \in J : k^{n-1} \leq
\beta_{\alpha} < k^n \right\}.$$ 

The following statement is a
simplification of Theorem~2 of \cite{MR2231044}, made possible by the properties of the
$2$-dimensional Lebesgue measure $\mu$.

\begin{thmKTV}
  Let $k$ be sufficiently large. Suppose there exists
  some $\theta \in \mathbb{R}_{>0}$ such that for any $ n \geq 1 $ and any rectangle $F_n $
  there exists a collection $\mathcal{C}(\theta F_n)$ of
  disjoint rectangles $2 \theta F_{n+1}$ contained within $\theta F_n$ such that
	\begin{equation}
		\label{eq:cond1}
		\# \mathcal{C}(\theta F_n) \geq \kappa_1 \frac{\mu \left(\theta F_n \right)}
		{\mu \left(\theta F_{n+1} \right)}
	\end{equation}
	and
	\begin{equation*}
    \# \left\{ 2 \theta F_{n+1} \subset \mathcal{C}(\theta F_n):
    R_{\alpha} \cap 2 \theta F_{n+1} \neq \emptyset \, \text{ for some }  \alpha \in J(n+1)\right\}
  \end{equation*}
  \begin{equation}
  \label{eq:cond2}
		\leq \kappa_2 \frac{\mu \left(\theta F_n \right)}{\mu \left(\theta F_{n+1} \right)},
	\end{equation}
	where $0< \kappa_2 < \kappa_1$ are absolute constants independent of $k$ and $n$.  Furthermore, suppose
  \begin{equation}
    \label{eq:cond3}
    \dim \left( \cup_{\alpha \in J} R_{\alpha} \right) < 2,
  \end{equation}
  then
  \begin{equation*}
    \dim \left( \Bad(\mathcal{R}, \beta, \rho_1, \rho_2) \right) = 2.
  \end{equation*}
\end{thmKTV}
\noindent
We can now prove Theorem~\ref{thm:dimbad}.
\begin{proof}[Proof of Theorem~\ref{thm:dimbad}]
	Fix two positive reals $i,j$ with $i+j=1$ and some $\mathbf{x} \in \Bad(i, j)$.
	It is once more assumed that $i, j >0$, for in this case the theorem
	would otherwise follow immediately from Corollary 1 of \cite{BHKV}.
	With reference to the above framework, set
	\begin{eqnarray*}
		& & J:=\left\{q \in	\mathbb{Z}_{\neq 0} \right\},	\quad \alpha:=q \in J, \quad R_{\alpha}:= R_q =
		\left\{ q \mathbf{x}+\mathbf{p}: \mathbf{p} \in \mathbb{Z}^2 \right\} \\
		& & \beta_{\alpha}:= \beta_q = \infabs{q}, \quad \rho_1(r):=1/r^i \quad \text{and} \quad \rho_2(r):=1/r^j.   \
	\end{eqnarray*}
	\noindent
	By design we then have
	\begin{equation*}
  	\Bad(\mathcal{R}, \beta, \rho_1, \rho_2) = \Bad^{\mathbf{x}}(i, j)
  \end{equation*}
	and so the proof is reduced to showing that the conditions of Theorem~KTV are satisfied.
	
	For $k > 1$ and $m \geq 1$, let  $F_m$ be a generic closed rectangle with centre in $[0,1]^2$ and
	of side lengths $2k^{-mi}$ and $2k^{-mj}$
	respectively . For $k $ sufficiently large	and any $\theta \in \mathbb{R}_{>0}$ it is
	clear that there exists a collection $\mathcal{C}(\theta B_m)$ of closed rectangles $2\theta F_{m+1}$ within
	$\theta F_m$ each of side lengths $4\theta k^{-(m+1)i}$ and $4\theta k^{-(m+1)j}$ respectively. Moreover, the
	number of rectangles in this collection exceeds
	\begin{equation*}
		\left\lfloor \frac{2\theta k^{-mi}}{4\theta k^{-(m+1)i}} \right\rfloor \times
		\left\lfloor \frac{2\theta k^{-mj}}{4\theta k^{-(m+1)j}} \right\rfloor.
	\end{equation*}
	Here, the notation $\left\lfloor \, . \, \right\rfloor$ denotes the integer part.
	For large enough $k$ the above is strictly positive and is bounded below by
	\begin{eqnarray*}
		\frac{1}{2} \left( \frac{2\theta k^{-mi}}{4\theta k^{-(m+1)i}} \right) \times
		\frac{1}{2} \left( \frac{2\theta k^{-mj}}{4\theta k^{-(m+1)j}} \right)
		& = & \frac{1}{16}\left( \frac{4\theta^2 k^{-m(i+j)}}{4\theta^2 k^{-(m+1)(i+j)}} \right) \\
		& = & \frac{1}{16} \frac{\mu(\theta F_m)}{\mu(\theta F_{m+1})}.  \
	\end{eqnarray*}
	Hence, inequality (\ref{eq:cond1}) holds with $\kappa_1: = 1/16$.
	
	We endeavour to show that the additional condition
	(\ref{eq:cond2}) on the collection $\mathcal{C}(\theta F_m)$ is
	satisfied.  To this end, we fix $m \geq 1$	and  proceed as follows.
	Choose two members of distinct moduli from the set $J(m+1)$; i.e., choose two integers $q$ and $q'$
	such that
  \begin{equation}
  	\label{eqn:help}
  	k^m \leq \infabs{q'}< \infabs{q} < k^{m+1}.
	\end{equation}
	Associated with the integers $q$ and $q'$ are the resonant sets $R_q$ and $R_{q'}$, whose elements take the form
	$q\mathbf{x}+\mathbf{p}$ and $q'\mathbf{x}+\mathbf{p'}$ respectively (for some $\mathbf{p}, \mathbf{p'}
	\in \mathbb{Z}^2$). Consider the minimum distance between a point in $R_q$ and one in $R_{q'}$. For $t=1,2$,
	\begin{eqnarray*}
		\infabs{(qx_t+p_t)-(q'x_t+p'_t)} & = & \infabs{(q-q')x_t +p_t-p'_t} \\
		& \geq & \norm{(q-q')x_t}. \
	\end{eqnarray*}
	Since $\mathbf{x} \in \Bad(i, j)$ either
	\begin{equation*}
		\norm{(q-q')x_1} \quad \geq
		\quad \left( \frac{c(\mathbf{x})}{\infabs{q-q'}} \right)^i
		\quad \stackrel{(\ref{eqn:help})}{>} \quad \left( \frac{c(\mathbf{x})}{2k^{m+1}} \right)^i
	\end{equation*}
	or
	\begin{equation*}
		\norm{(q-q')x_2} \quad \geq
		\quad \left( \frac{c(\mathbf{x})}{\infabs{q-q'}} \right)^j
		\quad \stackrel{(\ref{eqn:help})}{>} \quad \left( \frac{c(\mathbf{x})}{2k^{m+1}} \right)^j.	
	\end{equation*}
	Therefore, if we set
	\begin{equation*}
		\theta: = \frac{1}{2} \min \left\{ \left( \frac{c(\mathbf{x})}{2k} \right)^i,
		\left( \frac{c(\mathbf{x})}{2k} \right)^j \right\}
	\end{equation*}
	then the rectangle $\theta F_m$ has respective side lengths
	\begin{equation*}
		2 \theta k^{-mi} \quad = \quad \min\left\{ \left( \frac{c(\mathbf{x})}{2k} \right)^i,
		\left( \frac{c(\mathbf{x})}{2k} \right)^j \right\} k^{-mi} \quad
		\leq \quad \left( \frac{c(\mathbf{x})}{2k^{m+1}} \right)^i
	\end{equation*}
	and
	\begin{equation*}
		2 \theta k^{-mj} \quad = \quad \min\left\{ \left( \frac{c(\mathbf{x})}{2k} \right)^i,
		\left( \frac{c(\mathbf{x})}{2k} \right)^j \right\} k^{-mj} \quad
		\leq \quad \left( \frac{c(\mathbf{x})}{2k^{m+1}} \right)^j.
	\end{equation*}
	So, for any two integers $q, q'$ of distinct moduli in $J(m+1)$, if a member of $R_q$
	lies in $\theta F_m$ then no members of $R_{q'}$ may lie in $\theta F_m$.
	Only one point of $R_q$ may lie in $\theta F_m$ (since $\mu(\theta F_m)<1$) and so only two points over all
	possible resident sets may lie in any rectangle $\theta F_m$; those corresponding to $q$ and $-q$. Hence,
	\begin{equation*}
    \# \left\{ 2 \theta F_{m+1} \subset \mathcal{C}(\theta F_m):
    R_{q} \cap 2 \theta F_{m+1} \neq \emptyset \, \text{ for some }  q \in J(m+1)\right\} \, \leq \, 2,
  \end{equation*}
  which for large enough $k$ is certainly less than
  \begin{equation*}
		\frac{k}{32} \quad = \quad \frac{1}{32} \frac{\mu \left(\theta F_m \right)}
		{\mu \left(\theta F_{m+1} \right)}.
	\end{equation*}
	So, with $\theta$ as defined above and with $\kappa_2: = 1/32 < \kappa_1$, the collection
	$\mathcal{C}(\theta F_m)$ satisfies inequality (\ref{eq:cond2}).
	
	Finally, note that the family $\mathcal{R}$ of resonant sets takes the form of a countable number of countable
	sets and so
  \begin{equation*}
    \dim \left( \cup_{q \in J} \, R_q \right) \, = \, 0
  \end{equation*}
  and inequality (\ref{eq:cond3}) trivially holds. Thus, the conditions of Theorem~KTV are
	satisfied and the theorem follows.
\end{proof}

\section{Appendix}
\label{sec:Appendix}

We conclude the paper by proving a general result implying Theorems \ref{thm:fullmeasure} \& \ref{thm:multmetrical} as stated in the main body of the paper.
The result is an extension of
Cassels' inhomogeneous Khintchine-type theorem \cite[Chapter VII, Theorem II]{Cas57}.
The proof is a modification of Cassels' original argument and also borrows ideas from the
work of Gallagher.
\begin{thm}
	\label{thm:doublymetric}
	For any sequence $\left\{ A_q \right\}_{q \in \mathbb{N}}$ of measurable subsets of $[0,1)^d$ let $A$ denote the set of all pairs
	$(\mathbf{x}, \boldsymbol\gamma) \in [0,1)^d \times [0,1)^d$ for which
	there exists infinitely many $q \in \mathbb{N}$ and $\mathbf{p} \in \mathbb{Z}^d$ such that
	\begin{equation}
		\label{eqn:doublymetric}
		q\mathbf{x} - \boldsymbol\gamma - \mathbf{p} \, \in \, A_q.
	\end{equation}
	Then,
	\begin{equation*}
		\mu_{2d}(A): =
			\begin{cases}
				0, & \quad \displaystyle\sum_{r=1}^{\infty} \mu_d({A_r}) \, < \, \infty, \\
  			1, & \quad \displaystyle\sum_{r=1}^{\infty} \mu_d({A_r}) \, = \, \infty,
  		\end{cases}
	\end{equation*}
	where $\mu_s$ denotes $s$-dimensional Lebesgue measure.
\end{thm}
\begin{proof}
	We begin by considering the case in which the sum $\sum_{r=1}^{\infty} \mu_d({A_r})$ converges. Fix $\boldsymbol\gamma \in [0, 1)^d$. For each natural number $q$ a vector $\mathbf{x}$ satisfying (\ref{eqn:doublymetric})
	uniquely determines the integral vector $\mathbf{p}$ in such a way that $\infabs{\mathbf{p}} < q$. Therefore, the measure of the set of all $\mathbf{x} \in [0, 1)^d$ that satisfy
	(\ref{eqn:doublymetric}) for each $q$ is given by
	\begin{equation*}
		\mu_d \left( \bigcup_{\mathbf{p} \in \left[ 0,\right. \left. q \right)^d} \frac{\left(A_q \oplus \boldsymbol\gamma \right) \oplus \mathbf{p}}{q}\right) \quad = \quad
		\sum_{\mathbf{p} \in \left[ 0,\right. \left. q \right)^d} \mu_d \left(\frac{\left(A_q \oplus \boldsymbol\gamma \right) \oplus \mathbf{p}}{q}\right),
	\end{equation*}
	since the union is disjoint. The dilation property of $\mu_d$ yields that this is equivalent to
	\begin{equation*}
		q^{-d} \sum_{\mathbf{p} \in \left[ 0,\right. \left. q \right)^d} \mu_d \left(\left(A_q \oplus \boldsymbol\gamma \right) \oplus \mathbf{p}\right) \quad = \quad
		q^{-d} q^d \cdot \mu_d \left(A_q \oplus \boldsymbol\gamma \right) \quad = \quad \mu_d \left(A_q \right),
	\end{equation*}
	by the translational invariance of $\mu_d$. Now, if $\sum_{r=1}^{\infty} \mu_d({A_r}) < \infty$, then for any $\epsilon > 0$ the set of vectors satisfying (\ref{eqn:doublymetric}) for
	any $q \geq Q$ has measure at most $\sum_{q\geq Q}\mu_d({A_q}) < \epsilon$ for large enough $Q$. In particular, the set of $\mathbf{x}$ with infinitely many solutions to
	(\ref{eqn:doublymetric}) has measure at most $\epsilon$. This completes the proof of the convergence case.

	Let us now assume that the sum $\sum_{r=1}^{\infty} \mu_d({A_r})$ diverges. Define the function $\alpha_q: \mathbb{R}^d \rightarrow \mathbb{R}$ for each natural number $q$ as follows. Let
	\begin{equation*}
		\alpha_q(\mathbf{x}): \, = \,
		\begin{cases}
				1, & \quad \exists \, \mathbf{p} \in \mathbb{Z}^d \text{ s.t. } \mathbf{x}-\mathbf{p} \in A_q. \\
  			0, & \quad \text{otherwise}.
  		\end{cases}
	\end{equation*}
	It is clear that each $\alpha_q$ is measurable since it is equivalent to the characteristic function of a countable union of measurable sets in $\mathbb{R}^d$. Next, for every
	natural number $Q$ define the	function $\mathcal{A}_Q: [0, 1)^d \times [0, 1)^d \rightarrow  \mathbb{R}$ by
	\begin{equation*}
		\mathcal{A}_Q(\mathbf{x}, \boldsymbol\gamma): \, = \, \sum_{q \leq Q} \alpha_q(q\mathbf{x}-\boldsymbol\gamma).
	\end{equation*}
	We wish to verify that $\mathcal{A}_Q$ is measurable. To that end, we introduce the following lemma, which is a generalisation of a well known result in measure theory and
	follows via simple modification of the classical proof (see for example \cite[Chapter 2, Proposition 3.9]{SS}).
	\begin{lem}
		\label{lem:translation}
		If $f$ is a measurable function on $\mathbb{R}^d$ then it follows that the function $F_q(\mathbf{x}, \boldsymbol\gamma):=f(q\mathbf{x}- \boldsymbol\gamma)$ is measurable on $\mathbb{R}^d \times \mathbb{R}^d$
		for every natural number $q$.
	\end{lem}
	
	Since $\alpha_q$ is finite valued (and finite sums of finite valued measurable functions are measurable functions) Lemma \ref{lem:translation} implies that $\mathcal{A}_Q$ is indeed
	measurable on $[0, 1)^d \times [0, 1)^d$. Furthermore, by construction, it is apparent that $\mathcal{A}_Q(\mathbf{x}, \boldsymbol\gamma)$ is simply the number of natural $q$ with
	$q\leq Q$ such that
	\begin{equation*}
		q\mathbf{x} - \boldsymbol\gamma - \mathbf{p} \, \in \, A_q \quad \quad \quad \text{ for some } \mathbf{p} \in \mathbb{Z}^d.
	\end{equation*}
	Hence, to complete the proof of Theorem~\ref{thm:doublymetric} it suffices to show  $\mathcal{A}_Q(\mathbf{x}, \boldsymbol\gamma) \rightarrow \infty$ almost everywhere as $Q \rightarrow \infty$.
	We will hereafter consider $\mathcal{A}_Q$ as a random variable in a probability space with probability measure $\mu_d$.
	
	For any positive measurable function
	$f:[0,1)^d \times [0,1)^d \rightarrow \mathbb{R}_{\geq 0}^d$ we denote the \textit{expectation} of $f$ by
	\begin{equation*}
		E(f):= \displaystyle\int_{[0,1)^d} \int_{[0,1)^d} f(\mathbf{x}, \boldsymbol\gamma) \, d\mathbf{x} \, d\boldsymbol\gamma.
	\end{equation*}
	If the \textit{variance} $V(f):=E(f^2)-E(f)^2$ of $f$ is finite then the famous Paley-Zygmund inequality (see for example \cite[Ineq. II, p.8]{Kah})
	states that
	\begin{equation*}
		\mu_d\left( \left\{(\mathbf{x}, \boldsymbol\gamma): f(\mathbf{x}, \boldsymbol\gamma) \geq \epsilon E(f) \right\} \right)  \, \, \,  \geq \, \, \,  (1-\epsilon)^2\frac{(E(f))^2}{E(f^2)},
	\end{equation*}
	for any sufficiently small $\epsilon>0$. We will use this inequality to reach our desired conclusion.
	
	Before applying the Paley-Zygmund inequality to $\mathcal{A}_Q$ we must show that $V(\mathcal{A}_Q)$ is finite. It suffices to show that both $E(\mathcal{A}_Q)$ and
	$E((\mathcal{A}_Q)^2)$ are finite. To do this we require the following lemma \cite[Chapter VII, Lemma 3]{Cas57}.
	\begin{lem}[Cassels]
		\label{lem:cassels}
		Let $\alpha$ be a measurable function of period one of the variable $\mathbf{x} \in \mathbb{R}^d$. Then,
		\begin{equation*}
			\displaystyle\int_{[0,1)^d} \alpha(q\mathbf{x}+\boldsymbol\gamma) \, d\mathbf{x} \, \, = \, \, \int_{[0,1)^d} \alpha(\mathbf{x}) \, d\mathbf{x},
		\end{equation*}
		for any vector $\boldsymbol\gamma \in \mathbb{R}^d$ and any integer $q \neq 0$. 	
	\end{lem}
	
	We note that $\alpha_q$ is of period one and so
	\begin{eqnarray}
		E(\mathcal{A}_Q)  & = & \int_{[0,1)^d} \int_{[0,1)^d} \mathcal{A}_Q(\mathbf{x}, \boldsymbol\gamma) \, d\mathbf{x} \, d\boldsymbol\gamma \nonumber \\
		& =
		&		\sum_{q \leq Q} \int_{[0,1)^d} \int_{[0,1)^d} \alpha_q(q\mathbf{x}-\boldsymbol\gamma) \, d\mathbf{x} \, d\boldsymbol\gamma \nonumber \\
		& \stackrel{\text{Lem. \ref{lem:cassels}}}{=}
		& \sum_{q \leq Q} \int_{[0,1)^d} \int_{[0,1)^d} \alpha_q(\mathbf{x}) \, d\mathbf{x} \, d\boldsymbol\gamma \nonumber \\
		& =
		& \sum_{q \leq Q} \int_{[0,1)^d} \int_{[0,1)^d} \chi_{A_q}(\mathbf{x}) \, d\mathbf{x} \, d\boldsymbol\gamma \nonumber \\
		\label{eqn:expectation}
		& = & \sum_{q \leq Q} \mu_d(A_q), \
	\end{eqnarray}
	which is indeed finite. Further,
	\begin{eqnarray*}
		E((\mathcal{A}_Q)^2)  & = & \int_{[0,1)^d} \int_{[0,1)^d} (\mathcal{A}_Q(\mathbf{x}, \boldsymbol\gamma))^2 \, d\mathbf{x} \, d\boldsymbol\gamma  \\
		& = &		\sum_{q, r \leq Q} \int_{[0,1)^d} \int_{[0,1)^d} \alpha_q(q\mathbf{x}-\boldsymbol\gamma)\alpha_r(r\mathbf{x}-\boldsymbol\gamma) \, d\mathbf{x} \, d\boldsymbol\gamma  \\
		& = &		\sum_{q, r \leq Q} \int_{[0,1)^d} \int_{[0,1)^d} \alpha_{r-s}(-\boldsymbol\gamma')\alpha_r(s\mathbf{x}'-\boldsymbol\gamma') \, d\mathbf{x}' \, d\boldsymbol\gamma' , \
	\end{eqnarray*}
	via the change of variables $\mathbf{x}':=\mathbf{x}$, $\boldsymbol\gamma':=\boldsymbol\gamma-q\mathbf{x}$ and $s:=r-q$. Here, the range of
	$\mathbf{x}'$ and $\boldsymbol\gamma'$ can both be taken as $[0,1)^d$ since the function $\alpha_q$ is periodic.
	Let
	\begin{equation*}
		\mathcal{A}^{(r, s)}(\mathbf{x'}, \boldsymbol\gamma'): \, = \, \int_{[0,1)^d}\int_{[0,1)^d}\alpha_{r-s}(-\boldsymbol\gamma')\alpha_r(s\mathbf{x}'-\boldsymbol\gamma') \,
		d\mathbf{x}' \, d\boldsymbol\gamma'.
	\end{equation*}
	Then,
	if $r=q$ then $s=0$ and we have
		\begin{eqnarray*}
		\mathcal{A}^{(r, s)}(\mathbf{x'}, \boldsymbol\gamma') & = & \int_{[0,1)^d} \int_{[0,1)^d} \left( \alpha_q(-\boldsymbol\gamma') \right)^2
		\, d\mathbf{x}' \, d\boldsymbol\gamma' \\
		& =
		& \int_{[0,1)^d} \int_{[0,1)^d} \alpha_q(-\boldsymbol\gamma') \, d\mathbf{x}' \, d\boldsymbol\gamma' \\
		& = & \mu_d(A_{q}). \
	\end{eqnarray*}
	However, if $r\neq q$ then $s \neq 0$ and we get
	\begin{eqnarray*}
		\mathcal{A}^{(r, s)}(\mathbf{x'}, \boldsymbol\gamma') & = & \int_{[0,1)^d} \alpha_{r-s}(-\boldsymbol\gamma') d\mathbf{x}'
		\int_{[0,1)^d} \int_{[0,1)^d} \alpha_r(s\mathbf{x}'-\boldsymbol\gamma')	 \, d\mathbf{x}' \, d\boldsymbol\gamma' \\
		& \stackrel{\text{Lem. \ref{lem:cassels}}}{=} & \mu_d(A_{r-s})  \int_{[0,1)^d} \int_{[0,1)^d} \alpha_r(\mathbf{x}') \, d\mathbf{x}' \, d\boldsymbol\gamma' \\
		& = & \mu_d(A_{q})\mu_d(A_{r}).
	\end{eqnarray*}	
	These equivalences yield that
	\begin{eqnarray*}
		E((\mathcal{A}_Q)^2)  & = & \sum_{q, r \leq Q} \mathcal{A}^{(r, s)}(\mathbf{x'}, \boldsymbol\gamma') \\
		& = & \sum_{q\leq Q} \mu_d(A_{q}) + \sum_{\stackrel{q, r\leq Q:}{q\neq r}} \mu_d(A_{q})\mu_d(A_{r}) \\
		& \leq & \sum_{q\leq Q} \mu_d(A_{q}) + \left(\sum_{q\leq Q} \mu_d(A_{q}) \right)^2 \\
		& \leq & (1-\epsilon)^{-2} \left(\sum_{q\leq Q} \mu_d(A_{q}) \right)^2  \\
		& = & (1-\epsilon)^{-2} \left(E(\mathcal{A}_Q) \right)^2,\
	\end{eqnarray*}
	for any sufficiently small $\epsilon>0$ and large enough $Q$ (because $\sum_{q\leq Q} \mu_d(A_{q}) \rightarrow \infty$ as $Q \rightarrow \infty$
	by assumption). Note that the final bound is finite as required.	
	
	In view of the Paley-Zygmund inequality we have that
	\begin{equation*}
		\mu_d\left( \left\{(\mathbf{x}, \boldsymbol\gamma): \mathcal{A}_Q(\mathbf{x}, \boldsymbol\gamma) \geq \epsilon \sum_{q\leq Q} \mu_d(A_{q})
		\right\} \right)  \, \, \,  \geq \, \, \,  (1-\epsilon)^4 \, \, \, \geq \, \, \, 1-4\epsilon.
	\end{equation*}
	Finally, since $\mathcal{A}_Q$ increases
	monotonically with $Q$, we have that $\mathcal{A}_Q(\mathbf{x}, \boldsymbol\gamma) \rightarrow \infty$ in $[0,1)^d \times [0,1)^d$ except on a
	set of measure at most $4\epsilon$. This completes the proof as the choice of $\epsilon$ is arbitrary.	
\end{proof}

\subsection*{Acknowledgements}
\label{sec:Acknowledgements}
The author would like to thank Sanju Velani for his enthusiastic encouragement and support, and for introducing him to the problems at hand. He would also like to thank the referees for their useful remarks and Simon Eveson for his helpful suggestions. This research was funded by the ESPRC.

\end{document}